\documentclass[a4paper]{article}

\usepackage{tikz}
\usepackage{tikz-cd}
\usetikzlibrary{matrix,arrows, patterns, positioning, calc, intersections}

\usepackage[utf8]{inputenc}

\usepackage{microtype}
\usepackage{amscd}
\usepackage{amsmath,amssymb,amsfonts, accents}
\usepackage{xstring}
\usepackage{xcolor}
\usepackage[mathscr,mathcal]{euscript}
\usepackage{xypic}
\xyoption{rotate}
\usepackage[cmtip,all,2cell]{xy}
\entrymodifiers={+!!<0pt,\fontdimen22\textfont2>}
\usepackage{url}
\usepackage{latexsym}
\usepackage{amsthm}
\usepackage{multicol}
\usepackage{enumitem}
\usepackage[colorlinks=true, linkcolor={blue!50!black}, pdfhighlight=/O, ocgcolorlinks=true]{hyperref}
\usepackage{graphicx}
\usepackage{color}
\usepackage{mathtools}
\usepackage[colorinlistoftodos]{todonotes}
\usepackage[affil-it]{authblk}

\usepackage{adjustbox}

\usepackage[a4paper, total={5.5in,8.5in}]{geometry}

\numberwithin{equation}{section}
\theoremstyle{plain}
\newtheorem{theorem}[equation]{Theorem}
\newtheorem{lemma}[equation]{Lemma}
\newtheorem{proposition}[equation]{Proposition}
\newtheorem{corollary}[equation]{Corollary}

\theoremstyle{definition}

\newtheorem{definition}[equation]{Definition}

\newtheorem{example}[equation]{Example}
\newtheorem{remark}[equation]{Remark}

\setlist[enumerate]{label=(\arabic*), leftmargin=*}
\setenumerate{label=(\arabic*), leftmargin=*}
\setlist[itemize]{label=$\vcenter{\hbox{\footnotesize$\bullet$}}$, leftmargin=*}
\setitemize{label=$\vcenter{\hbox{\footnotesize$\bullet$}}$, leftmargin=*}

\newcommand{\mb}[1]{\mathbf{#1}}
\newcommand{\mm}[1]{\mathrm{#1}}

\newcommand{\ul}[1]{\underline{#1}}

\newcommand{\cat}[1]{
\StrLen{#1}[\mystrlen]
\ifnum\mystrlen=1 \mathscr{#1}
\else \mathrm{#1}
\fi}
\newcommand{\scat}[1]{\mb{#1}}

\newcommand{\Hom}[0]{\mm{Hom}}

\newcommand{\op}[0]{\mm{op}}

\newcommand{\sarahline}[1]{\todo[inline,color=green!40]{#1}}

\newcommand{\tristan}[1]{\todo{#1}}
\newcommand{\tristanline}[1]{\todo[inline]{#1}}

\title{Calabi--Yau structures for multiplicative preprojective algebras}

\author[]{
	Tristan Bozec\thanks{IMAG, Univ. Montpellier, CNRS, Montpellier, France \\
											\href{mailto:tristan.bozec@umontpellier.fr}{tristan.bozec@umontpellier.fr}}, 
	Damien Calaque\thanks{IMAG, Univ. Montpellier, CNRS, Montpellier, France \\
											\href{mailto:damien.calaque@umontpellier.fr}{damien.calaque@umontpellier.fr}}, 
	Sarah Scherotzke\thanks{Mathematical Institute, University of Luxembourg, Luxembourg \\
											\href{mailto:sarah.scherotzke@uni.lu}{sarah.scherotzke@uni.lu}}}

\date{}

\begin{document}

\maketitle

\begin{abstract}
In this paper we deal with Calabi--Yau structures associated with (differential graded versions of) deformed multiplicative preprojective algebras, of which we provide concrete algebraic descriptions. 
Along the way, we prove a general result that states the existence and uniqueness of negative cyclic lifts for non-degenerate relative Hochschild classes. 
\end{abstract}

\setcounter{tocdepth}{2}
\tableofcontents


\section{Introduction}

Given a quiver $Q$, that is an oriented graph, one may consider the preprojective algebra associated to $Q$ over some field $k$. It can be defined as a quotient of the path algebra $k\overline Q$ 
of the double quiver $\overline Q$, obtained by adjoining to each edge $e:i\to j$ between two vertices $i$ and $j$ a reverse edge $e^*:j\to i$. We quotient by a single relation\[
\mu=\sum_{e\in Q}[e,e^*],\]
a signed combination of $2$-cycles in $\overline Q$. Originally introduced by Gelfand--Ponomarev~\cite{GePo} (see also~\cite{Ring}) in a strictly algebraic context, the preprojective algebra turns out 
bearing a strong geometric significance. It may indeed be understood as the algebraic structure underlying the cotangent to the moduli of representations of $Q$ (see~\cite{BCS} for a fully precise statement). 
Its representations correspond to the $0$-fiber of the moment map associated to the linear group acting by conjugation at each vertex. This fact has been extensively used by Lusztig~\cite{Lusz}, and later 
Nakajima~\cite{Nak}, to geometrically realize quantum groups and their representations, in particular through the definition of lagrangian subvarieties of symplectic quiver moduli. 

Multiplicative variants of these preprojective algebras have been introduced by Crawley-Boevey and Shaw~\cite{CBS} in the course of their study of the Deligne--Simpson problem. It is defined by performing a 
quotient by
\[
\prod_{e\in Q}(1+ee^*)(1+e^*e)^{-1}
\]
of an appropriate localization of $k\overline Q$. These variants turn out to naturally appear in various areas such as character varieties~\cite{Boalch,CBS,ScTi,Yam}, local systems on Riemann surfaces and perverse sheaves on nodal curves~\cite{Beil,BK,BoalchAnnals,Yam}, 
or integrable systems~\cite{ChaFai} among others.
The geometric framework in which multiplicative quiver varieties seem to be better studied is the one of quasi-hamiltonian reduction and group-valued moment maps from~\cite{AMM}, as shown by Van den Bergh~\cite{VdB,VdB2}.

Multiplicative preprojective algebras fit both into the quasi-hamiltonian formalism and into its non-commutative analogue as developped by Van den Bergh in~\cite{VdB2}. 
In the first case, group-valued moment maps and the quasi-hamiltonian formalism have a nice interpretation within the framework of shifted symplectic geometry of~\cite{PTVV}, in terms of lagrangian morphisms and derived lagrangian 
intersections (see~\cite{CalTFT,Safronov}). 

Using the non-commutative analogue of quasi-hamiltonian formalism,  one obtains that multiplicative preprojective algebras come equipped with double quasi-Poisson structures~\cite{VdB}. Furthermore, Fern\'andez and Herscovich have 
recently proved in~\cite{FH} that double quasi-Poisson structures give rise to pre-Calabi--Yau structures in the sense of Iyudu--Kontsevich--Vlassopoulos~\cite{IKV}\footnote{We would like to warn the reader that 
pre-Calabi--Yau structures in \textit{loc.~cit.} are different from the ones considered e.g.~in~\cite{BCS,ToCY} (the latter being non-commutative pre-symplectic strctures, rather than non-commutative Poisson structures). }, 
extending a similar result from \textit{loc.~cit.} for double Poisson structures. 

In the same way as shifted Poisson structures in the sense of~\cite{CPTVV,Prid} arise on the source of morphisms equipped with a lagrangian structure~\cite{MS2} (actually, shifted Poisson structures are conjectured to be equivalent to lagrangian thickenings), 
it is expected that pre-Calabi--Yau structures in the sense of~\cite{IKV} often (if not always) arise on the target of Calabi--Yau morphisms in the sense of~\cite{BD1}. 

Hinging on these observations, our goal is to directly construct Calabi--Yau structures on appropriate algebraic objects, and get back the usual lagrangian morphisms associated 
with group-valued moment maps on moduli spaces. Namely, we do the following:\begin{enumerate}
\item We first study a Calabi--Yau structure on $k[x^{\pm1}]$, seen as the multiplicative analog of $k[x]$.  Using this Calabi-Yau structure, we obtain using ~\cite{BD1} a $1$-shifted symplectic structure on moduli of perfect complexes. We show  that restricting to the moduli of representations we recover the usual $1$-shifted symplectic structure on the adjoint quotient which is crucial in the derived symplectic interpretation of the quasi-hamiltonian formalism;
\item We give a $1$-Calabi-Yau structure on cospans which allow us to retrieve standard lagrangian correspondences when applying the moduli of objects functor $\mathbf{Perf}$. 
We recover in particular the lagrangian correspondence that was shown in~\cite{Safronov} to underlie the fusion product  from~\cite{AMM};
\item We give a relative $1$-Calabi--Yau structure on the algebraic counterpart of the group-valued moment map. This is done via a gluing procedure called fusion. 
\item Via pushouts of Calabi-Yau cospans, we obtain a $2$-Calabi-Yau structure on the differential graded multiplicative preprojective algebra, defined in Theorem~\ref{dgmult}. The zero truncation
of the  differential graded multiplicative preprojective algebra is the original multiplicative preprojective algebra. 
\end{enumerate}

\subsection*{Description of the paper}

In section~\ref{section2}, we provide a short recollection on Calabi--Yau structures, after~\cite{BD1}. We also show that, in the case of smooth dg-categories sitting in degree $0$, the required cyclic lift of the non-degenerate relative Hochschild class, in the definition of a 
Calabi--Yau structure on a morphism, automatically exists and is unique. This extends to the relative case a result of~\cite{TdV-VdB}, and is of independent interest, see Theorem~\ref{theorem-astuce}. 

Section~\ref{section: CY} uses this result to produce $1$-Calabi--Yau structures on $k[\mathbb G_m]=k[x^{\pm1}]$ and the cospan defined by $k[x^{\pm1}]\amalg k[y^{\pm1}]\rightarrow k\langle x^{\pm1},y^{\pm1}\rangle \leftarrow k[z^{\pm1}]$, denoted by $\mathcal F$ in this introduction only. 
In each case we define explicit Hochschild classes that we prove to be non-degenerate. Thanks to section~\ref{section2}, these admit a unique cyclic lift. We also study evaluation morphisms $k[x^{\pm1}]\to k$.

In section~\ref{section-preproj}, using the Calabi-Yau structures of the previous section, we show in Theorem~\ref{fuz} that the multiplicative moment map is $1$-Calabi-Yau. The quiver $A_2=\bullet\rightarrow\bullet$ serves as a building block.  Again, this structure is made 
explicit and proven to be non-degenerate ``by hand'', whereas its cyclic lift exists thanks to section~\ref{section2}. The cospan $\mathcal F$ studied earlier then serves in a gluing process (\textit{a.k.a.\! fusion}) to extend our result to arbitrary quivers. Pushouts along evaluation morphisms yield a $2$-Calabi-Yau structure on a dg-algebra whose $0$-truncation is the classical multiplicative preprojective algebra, \textit{c.f.}\ Theorem~\ref{dgmult}.

Finally, section~\ref{section-comparison} justifies our choices of Hochschild classes defining Calabi--Yau structures. We prove that when taking $\mathbf{Perf}$, we retrieve standard symplectic structures. Namely, the $1$-shifted symplectic structure on $\mathbf{Perf}_{k[x^{\pm1}]}$ matches the one on the derived loop stack $\mathcal L\mathbf{Perf}_k$. We also prove that the Calabi--Yau structure on $\mathcal F$ corresponds to a particular gluing of the boundaries of the  pair-of-pants. We conjecture that the structures we get on our dg-variants of multiplicative preprojective algebras yield standard quasi-hamiltonian structures on multiplicative quiver varieties.

\subsection*{Related works}
In~\cite{BK}, Bezrukavnikov and Kapranov prove that certain triangulated categories of microlocal complexes on nodal curves have a Calabi--Yau property, which roughly corresponds to the existence of an 
almost Calabi--Yau structure according to our terminology. However, it is not clear if it admits an actual Calabi--Yau (i.e.~cyclic) lift. In \textit{loc.~cit.} the authors mention a dg-version of the multiplicative 
preprojective algebra and expect that it is a Calabi--Yau dg-algebra. Our results show that this is indeed true. 
This expectation was motivated by the existence of an equivalence of abelian categories between microlocal sheaves on nodal curves with rational components on the one side, and modules over the multiplicative 
preprojective algebra on the other side. 
They could not conclude, because it is not known if a similar equivalence holds for the dg-version of the multiplicative preprojective algebra. 

A similar approach is considered by Shende and Takeda in their work~\cite{ST} on Calabi--Yau structures of topological Fukaya categories. It is possible that, following some suggestions from \cite[\S7.4 \& \S7.5]{ST}, 
one could potentially recover some of our existence results in a geometric way, as opposed to the explicit algebraic approach of the present paper. The obtained Calabi--Yau structures would then deserve to be compared 
with ours. The last section of our paper provides tools for such a comparison, that would then essentially rely on the fact that both approaches are compatible with some gluing/fusion process; though, the comparison for 
the $A_2$ building bloc remains to be dealt with. 

Yeung~\cite[\S5.5]{Yeung} also exhibits a Calabi--Yau structure on the multiplicative preprojective dg-algebra associated with a star-shaped quiver. It seems that his Calabi--Yau structure differs from ours, unless one goes 
through some adic completion (see Remark~\ref{remyeung}). 

Finally, Kaplan and Schedler~\cite{KaSc} show that the multiplicative preprojective (non dg) algebra is Calabi--Yau whenever the quiver is connected and contains an unoriented cycle. They conjecture that it is still Calabi--Yau 
for every quiver that is connected and not Dynkin; they also conjecture that, under these assumptions, the multiplicative preprojective algebra and its dg-version are quasi-isomorphic (and they prove this for quivers satisfying 
a ``strong free product property''). This is consistent with our result, which only concerns the dg-version, but holds without any assumption on the quiver. 

\subsection*{Acknowledgements}
The first and second author have received funding from the European Research Council (ERC) under the
European Union's Horizon 2020 research and innovation programme (Grant Agreement No.\ 768679).


\section{Calabi--Yau structures}\label{section2}

\subsection{Recollection on Calabi--Yau structures}

Along this paper, we use the same notation and terminology as in \cite{BCS}. We recall briefly the most important information.  Note also that in this paper $k$ is a field of characteristic zero.

\medskip

We denote by  $\cat{Mod}_k$ the category of cochain complexes over $k$. A \textit{dg-category} is a $\cat{Mod}_k$-enriched category and  the category of dg-categories with dg-functors is denoted by $\cat{Cat}_k$. 
We refer to \cite{KellerDG,ToDG2} for a detailed introduction to dg-categories and their homotopy theory.  

If $\cat{M}$ is a model category, we will write $\scat{M}$ for the corresponding 
$\infty$-category obtained by localizing along weak equivalences. 

We use the notation ``$\cat{Map}$'' to distinguish the \textit{space} of $\infty$-categorical morphisms from the \textit{set} 
of $1$-categorical morphisms, for which we use the notation ``$\cat{Hom}$''. The underlined versions designate their enriched 
counterparts (unless otherwise specified, the enrichment is over complexes). 
If a category has a symmetric monoidal structure which is closed, we use upper case letters for the internal enrichment, i.e.\
$\textsc{Hom}$ and $\textsc{Map}$, for categories and $\infty$-categories, respectively. 

\medskip

Recall the \textit{Hochschild chains} functor 
\[
\cat{HH}\,:\,\scat{Cat}_k\longrightarrow \scat{Mod}_k\,;\,
\cat{A}\longmapsto \cat{A}\underset{\cat{A}^e}{\overset{\mathbb{L}}{\otimes}}\cat{A}^{\op}\,,
\]
where $\cat{A}^e:=\cat{A}\otimes\cat{A}^{\op}$. We write $\cat{HH}_{i}(\cat{A})$ for the $(-i)$-th cohomology of $\cat{HH}(\cat{A})$. 
Dually, $\cat{HH}^i(\cat{A})$ is defined as the $i$-th cohomology of $\mathbb{R}\ul{\cat{Hom}}_{{\cat{A}^e}} (\cat{A}, \cat{A})$. 

\medskip

Hochschild chains carry a mixed structure, which is given in the standard explicit model by Connes's $B$-operator. 
The negative cyclic complex of $\cat{A}$, denoted by $\cat{HC}^-(\cat{A})$, is defined as the homotopy fixed points of $\cat{HH}(\cat{A})$ with respect to the mixed structure; 
it comes with a natural transformation $(-)^\natural: \cat{HC}^- \Rightarrow \cat{HH}$. 
As before, $\cat{HC}^-_{i}(\cat{A})$ stands for the $(-i)$-th cohomology of $\cat{HC}^-(\cat{A})$. 

Recall the \textit{inverse dualizing functor} \[
(-)^\vee\,:\,\scat{Mod}_{\cat{A}^e}\longrightarrow \scat{Mod}_{\cat{A}^e}^{\op}
\]
that is given as follows: for a right $\cat{A}^e$-module $\cat{M}$, and an object $a\in\mathrm{Ob}(\cat{A}^{\op}\otimes\cat{A})$, 
\[
\cat{M}^\vee(a):= \mathbb{R}\ul{\cat{Hom}}_{\cat{Mod}_{(\cat{A}^e)^{\op}}}\big(\cat{M}\circ\tau,\cat{A}^e(a,-)\big)\,.
\]
where $\tau$ is the anti-involution $\tau:\cat{A}^e\tilde\longrightarrow(\cat{A}^e)^{\op}$ such that $\cat{A}\circ\tau=\cat{A}^{\op}$.

A dg-category $\cat{A}$ is \emph{smooth} if $\cat{A}$ is a perfect $\cat{A}^e$-module. For smooth dg-categories, we have the following equivalences 
\[
(-)^\flat: \cat{HH}( \cat{A}) \stackrel{\sim} \longrightarrow    \mathbb{R}\ul{\cat{Hom}}_{\cat{Mod}_{\cat{A}^e}} (\cat{A}^\vee, \cat{A})
\]
and 
\[
\mathbb{R}\ul{\cat{Hom}}_{\cat{Mod}_{\cat{A}^e}} (\cat{A}, \cat{A}) \simeq   \cat{A}^\vee \overset{\mathbb{L}}{\underset{\cat{A}^e}{\otimes}}\cat{A}\,. 
\]

\begin{definition}\label{def:CY}
Let $\cat{A}$ be a  smooth dg-category. 
\begin{enumerate}
\item A class $c:k[n]\to\cat{HH}(\cat{A})$ such that $c^\flat:\cat{A}^\vee[n]\to \cat{A}$ is an equivalence is called \textit{non-degenerate}. 
Such a non-degenerate Hochschild class is called an \emph{almost $n$-Calabi--Yau} structure on $\cat{A}$. 
\item A \emph{$n$-Calabi--Yau structure} on $\cat{A}$ is a class $c:k[n]\to \cat{HC}^-(\cat{A})$ such that $c^\natural$ 
is non-degenerate. 
\end{enumerate}
\end{definition}

We now recall (relative) Calabi--Yau structures on morphisms and cospans of dg-categories, following Brav--Dyckerhoff~\cite{BD1} and To\"en~\cite[\S5.3]{ToEMS}. 

\begin{definition}
Let $\cat{A}\overset{f}{\longrightarrow}\cat{C}\overset{g}{\longleftarrow}\cat{B}$ be a cospan of smooth dg-categories. 
\begin{enumerate}
\item An \emph{almost $n$-Calabi--Yau structure} on this cospan is the data of a homotopy commuting diagram 
\[
\xymatrix{
k[n] \ar[r]^{c_{\cat{B}}}\ar[d]_{c_{\cat{A}}}  & \cat{HH}(\cat{B}) \ar[d] \\
\cat{HH}(\cat{A}) \ar[r] & \cat{HH}(\cat{C})
}\] 
such that $c_{\cat{A}}$ and $c_{\cat{B}}$ are non-degenerate in the sense of Definition \ref{def:CY}(1), and such that the homotopy 
$\cat{HH}(f)(c_{\cat{A}})\sim \cat{HH}(g)(c_{\cat{B}})$ is non-degenerate in the following sense: the induced (homotopy) commuting square 
\[
\xymatrix{
\cat{C}^\vee[n]\ar[r]^-{g^\vee}\ar[d]_-{f^\vee} 
& (\cat{B}^\vee[n])\overset{\mathbb{L}}{\underset{\cat{B}^e}{\otimes}}\cat{C}^e 
\overset{c_{\cat{B}}^\flat\otimes\mathrm{id}}{\simeq} \cat{B}\overset{\mathbb{L}}{\underset{\cat{B}^e}{\otimes}}\cat{C}^e \ar[d]^-{g\otimes\mathrm{id}} \\
(\cat{A}^\vee[n])\overset{\mathbb{L}}{\underset{\cat{A}^e}{\otimes}}\cat{C}^e 
\overset{c_{\cat{A}}^\flat\otimes\mathrm{id}}{\simeq} \cat{A}\overset{\mathbb{L}}{\underset{\cat{A}^e}{\otimes}}\cat{C}^e \ar[r]^-{f\otimes\mathrm{id}}
& \cat{C}
}
\]
is cartesian. 
\item An $n$-Calabi--Yau structure on the cospan is a homotopy 
commuting diagram 
\[
\xymatrix{
k[n] \ar[r]^{c_{\cat{B}}}\ar[d]_{c_{\cat{A}}}  & \cat{HC}^-(\cat{B}) \ar[d] \\
\cat{HC}^-(\cat{A}) \ar[r] & \cat{HC}^-(\cat{C})
}\]
such that the image under $(-)^{\natural} $ is an almost $n$-Calabi--Yau structure. 
\item If $\cat{A}=\varnothing$ then we call these (almost) $n$-Calabi--Yau structures on the morphism $g$. 
\end{enumerate}
\end{definition}
Recall that $n$-Calabi--Yau cospans do compose: after \cite[Theorem 6.2]{BD1}, the non-degeneracy property is preserved under composition. 

We finally note that whenever $\cat{A}=\cat{B}=\varnothing$, an $n$-Calabi--Yau structure on $\varnothing\rightarrow\cat{C}\leftarrow\varnothing$ is the same as an $(n+1)$-Calabi--Yau structure on $\cat{C}$. 
In particular, the push-out of two $n$-Calabi--Yau morphisms automatically inherits an $(n+1)$-Calabi--Yau structure. 

\subsection{Existence and uniqueness of cyclic lifts}

\begin{proposition}[\cite{TdV-VdB}, Section 5]\label{prop: lift}
Suppose $\cat{B}$ is a smooth dg-category. If $\cat{B}$ is almost $n$-Calabi--Yau, then $\cat{HH}^i(\cat{B}) \simeq \cat{HH}_{n-i}(\cat{B})$ for every $i\in\mathbb{Z}$. 
Furthermore, if $\cat{B}$ is concentrated in degree zero then
\begin{itemize}
\item[(a)] $\cat{HH}_i(\cat{B})=0 $ for all $i \not=0, 1, \ldots, n$;
\item[(b)] $\cat{HC}_i^-(\cat{B})=0 $ for all $i>n$; 
\item[(c)] the natural map $\cat{HC}_n^-(\cat{B})\to \cat{HH}_n(\cat{B})$ is an isomorphism. 
\end{itemize}
In particular, every almost $n$-Calabi--Yau structure on $\cat{B}$ admits an $n$-Calabi--Yau lift. 
\end{proposition} 
\begin{proof}
This is essentially \cite[Proposition 5.5, Corollary 5.6 \& Proposition 5.7]{TdV-VdB}. We reproduce the proof here for the reader's convenience. 

We have an isomorphism of $\cat{B}^e$-modules $c^\flat: \cat{B}^\vee[n] \simeq \cat{B}$. It yields 
\[
\cat{HH}_{n-i}(\cat{B})\simeq \Hom_{\mathrm{Ho}(\scat{Mod}_{\cat{B}^e})}(\cat{B}^\vee[n], \cat{B}[i]) \overset{c^\flat}{\simeq}  \Hom_{\mathrm{Ho}(\scat{Mod}_{\cat{B}^e})}(\cat{B}, \cat{B}[i])\simeq \cat{HH}^i(\cat{B})\,.
\]
If $\cat{B}$ is concentrated in degree zero then its Hochschild homology and cohomology are concentrated in non-negative degrees, and thus, using the above identifications, 
$\cat{HH}_i(\cat{B})=0 $ for all $i \not=0, 1, \ldots, n$. 
We then consider the negative cyclic complex, which is given by taking formal power series in a degree $2$ variable $u$ with coefficients in the Hochschild complex, and differential being given as $d-u\delta$, 
where $\delta$ is the mixed differential. The first page of the spectral sequence associated with the 
filtration by powers of $u$ reads as follows: 
\[\xymatrix { 0 & u\cat{HH}_0(\cat{B}) \ar[r]^-{u\delta} & u^2\cat{HH}_1(\cat{B})\ar[r]^-{u\delta}  & \cdots & u^{n+1}\cat{HH}_n(\cat{B}) \\
\cat{HH}_0(\cat{B}) \ar[r]^-{u\delta}& u\cat{HH}_1(\cat{B}) \ar[r]^-{u\delta}& \cdots  & u^n\cat{HH}_n(\cat{B}) & 0  \\
\cat{HH}_1(\cat{B}) \ar[r]^-{u\delta} & u\cat{HH}_2(\cat{B}) \ar[r]^-{u\delta} & \cdots & 0 \\
\cdots &&&\\
\cat{HH}_{n}(\cat{B}) & 0 & 0& \cdots 
}\]
This proves (a) and (b). 
\end{proof} 
\begin{remark}\label{remark-prop}
Under the assumption of Proposition~\ref{prop: lift}, the duality isomorphism extends to Hochschild homology with values in any $\cat{B}$-bimodule $\cat{M}$: $\mathrm{H}^i(\cat{B},\cat{M})\simeq \mathrm{H}_{n-i}(\cat{B},\cat{M})$. 
Moreover, if both $\cat{B}$ and $\cat{M}$ are concentrated in degree $0$, then (a) still holds: $\mathrm{H}_{i}(\cat{B},\cat{M})$ vanishes for all $i\neq0,\dots,n$. 
\end{remark}

\begin{theorem}\label{theorem-astuce}
Let $F: \cat{B} \to \cat{C}$ be a functor between smooth dg-categories that are concentrated in degree zero. 
Every almost $n$-Calabi--Yau structure on $F$ admits a unique $n$-Calabi--Yau lift. 
\end{theorem}
\begin{proof}
First of all, we know from Proposition~\ref{prop: lift} that the almost $n$-Calabi--Yau structure $c_{\cat{B}}\in\cat{HH}_n(\cat{B})$ on $\cat{B}$ uniquely lifts to a $n$-Calabi--Yau structure $c_{\cat{B}}^-\in\cat{HC}^-_n(\cat{B})$. 
The other part of the almost $n$-Calabi--Yau structure on $F$ is a homotopy from $F(c_{\cat{B}})$ to $0$, which amounts to the choice of a relative lift $c_F\in \cat{HH}_{n+1}(\cat{C},\cat{B})$ of $c_{\cat{B}}$. 
Indeed, $\cat{HH}_i(\cat{C},\cat{B})$ is defined as the $(-i)$-th cohomology of the homotopy cofiber (or, mapping cone) of $\cat{HH}(\cat{B})\to \cat{HH}(\cat{C})$, so that we have a long exact sequence 
\[
\cdots\to \cat{HH}_{n+1}(\cat{C},\cat{B})\to \cat{HH}_n(\cat{B})\to \cat{HH}_n(\cat{C})\to \cat{HH}_{n}(\cat{C},\cat{B})\to\cdots
\] 
The non-degeneracy of $c_F$ tells us that the nul-homotopic sequence of $\cat{C}^e$-modules
\[
\cat{C}^\vee[n] \to \cat{B}^\vee[n] \otimes_{\cat{B}^e} \cat{C}^e \simeq \cat{B}\otimes_{\cat{B}^e} \cat{C}^e \to \cat{C}
\]
is actually a homotopy fiber sequence. Applying $\Hom_{\mathrm{Ho}(\scat{Mod}_{\cat{C}^e})}(-,\cat{C})$ yields a long exact sequence 
\[
\cdots\to \cat{HH}^k(\cat{C})\to \mathrm{H}^k(\cat{B},\cat{C})\simeq \mathrm{H}_{n-k}(\cat{B},\cat{C})\to \cat{HH}_{n-k}(\cat{C})\to \cat{HH}^{k+1}(\cat{C})\to \cdots
\]
Hence, using that Hochschild homology and cohomology of $\cat{C}$ vanishes for negative indices (because $\cat{C}$ is concentrated in degree $0$), together with the version from remark~\ref{remark-prop} of the vanishing property (a), we get that 
the Hochschild homology (and cohomology) of $\cat{C}$ vanishes in degrees $i \not= 0, \ldots, n+1$.
We again look at the first page of the Hochschild-to-negative cyclic spectral sequence: 

\[\xymatrix { 0 & u\cat{HH}_0(\cat{C}) \ar[r]^-{u\delta} & u^2\cat{HH}_1(\cat{C})\ar[r]^-{u\delta}  & \cdots & u^{n+2}\cat{HH}_{n+1}(\cat{C}) \\
\cat{HH}_0(\cat{C}) \ar[r]^-{u\delta}& u\cat{HH}_1(\cat{C}) \ar[r]^-{u\delta}& \cdots  & u^{n+1}\cat{HH}_{n+1}(\cat{C}) & 0  \\
\cat{HH}_{1}(\cat{C}) \ar[r]^-{u\delta} & u\cat{HH}_2(\cat{C}) \ar[r]^-{u\delta} & \cdots & \\
\cdots &&&\\
\cat{HH}_n(\cat{C}) \ar[r]^-{u\delta} &  u\cat{HH}_{n+1}(\cat{C})& 0 &0 \\
~\cat{HH}_{n+1}(\cat{C}) & 0 & 0& 
}\]

Putting this together, we obtain the following morphism of exact sequences:  
\[\xymatrix{ 0\ar[r]& \cat{HC}^-_{n+1}(\cat{C}) \ar[d]^[left]{\sim}  \ar[r] & \cat{HC}^-_{n+1}(\cat{C},\cat{B}) \ar[d] \ar[r] & \cat{HC}^-_{n}(\cat{B}) \ar[d]^[left]{\sim} \ar[r]  & \cat{HC}^-_n(\cat{C}) \ar@{>->}[d] \\
0\ar[r]& \cat{HH}_{n+1}(\cat{C})  \ar[r] & \cat{HH}_{n+1}(\cat{C},\cat{B})  \ar[r] &\cat{HH}_n(\cat{B}) \ar[r]  & \cat{HH}_n(\cat{C} )}
\]

The injectivity of the rightmost arrow follows from the fact that \[
\cat{HC}^-_n(\cat{C}) \simeq \ker\big( \cat{HH}_n(\cat{C}) \to u\cat{HH}_{n+1}(\cat{C})\big)\,,\]
and it implies that the image of $c_{\cat{B}}^-$ via $\cat{HC}^-_n(\cat{B}) \to \cat{HC}^-_n(\cat{C})$ vanishes (because the image of $c_{\cat{B}}$ through $\cat{HH}_n(\cat{B}) \to \cat{HH}_n(\cat{C})$ does so). 
Therefore $c_{\cat{B}}^-$ lifts to a relative class in $\cat{HC}^-_{n+1}(\cat{C},\cat{B})$. The map from the affine space of relative lifts of $c_{\cat{B}}^-$ to the affine space of relative lifts of $c_{\cat{B}}$ is affine and modelled on the 
linear map $\cat{HC}^-_{n+1}(\cat{C})\to\cat{HH}_{n+1}(\cat{C})$, which is an isomorphism. Using that both affine spaces are non-empty, we get that the map from relative lifts of $c_{\cat{B}}^-$ to relative lifts of $c_{\cat{B}}$ is a bijection. 
Hence we get that a cyclic lift of $c_F$ exists and is unique. 
\end{proof}


\section{Calabi--Yau structures associated with $k[x^{\pm1}]$} \label{section: CY}

\subsection{A Calabi--Yau structure on $k[x^{\pm1}]$}

Let $\mathcal A=k[x^{\pm1}]=k[\mathbb{G}_m]$. It is the function ring of a smooth affine algebraic variety; hence 
$1$-Calabi--Yau structures on $\mathcal A$ are exactly non-vanishing top degree (here, degree $1$) forms. 
The Calabi--Yau structure we consider on $\mathcal A$ is, up to a scalar, $\alpha:=d_{dR}\log(x)=x^{-1}d_{dR}x$. In the rest of this subsection, 
we provide descriptions of this $1$-Calabi--Yau structure that will be convenient for later purposes. 

\begin{remark}\label{rem:inv}
Notice that the inverse morphism $inv:x\mapsto x^{-1}$ allows to identify $(\mathcal A,\alpha)$ with $(\mathcal A,-\alpha)$.  
We also observe that $\alpha$ is invariant under rescaling maps $x\mapsto q x$, $q\in k^\times$ (i.e.~it is of zero 
weight for the action of $\mathbb{G}_m$ on itself by multiplication). 
\end{remark}

\begin{remark}\label{remyeung}
In \cite{Yeung}, Yeung also considers a Calabi--Yau structure on $k[z^{\pm1}]$, which is \emph{different} form ours: Yeung's Calabi--Yau structure is given by $d_{dR}z$, and is exact, as opposed to ours. 
On moduli of representations, the Calabi--Yau structure we consider gives back the $1$-shifted symplectic structure that encodes the quasi-hamiltonian formalism (see Section \ref{section-comparison} below);
we expect that Yeung's Calabi--Yau structure rather leads to a linearized version of it. As a matter of fact, if one considers the $\mathcal I$-adic completion $\hat{\mathcal A}$ at the kernel $\mathcal I$ of the evaluation at $x=1$, 
the morphism $k[z]\to \hat{\mathcal A}$ sending $z$ to $\log(x)$ is well-defined and sends the canonical Calabi--Yau structure $d_{dR}z$ on $k[z]$ to $\alpha$. 
\end{remark}

\subsubsection{The cyclic cycle}

We work with the normalized Hochschild complex $C_n(\mathcal A)=\mathcal A\otimes\bar {\mathcal A}^{\otimes n}$ where $\bar{ \mathcal A}=\mathcal A/k$, with Hochschild differential $b$. On $C_n(\mathcal A)$, the 
Connes boundary map is given by \[
B(x_0\otimes\dots\otimes x_n)=\sum_{i=0}^n(-1)^{ni}1\otimes x_i\otimes\dots\otimes x_n\otimes x_0\otimes\dots\otimes x_{i-1}.\]
 We set
\[
2\alpha_n=(x^{-1}\otimes x)^{\otimes n}-(x\otimes x^{-1})^{\otimes n}\in C_{2n-1}(\mathcal A),
\]
so that
\[
b(\alpha_n)=2(1\otimes\alpha_{n-1})\qquad\textrm{and}\qquad
B(\alpha_n)=2n(1\otimes \alpha_n).
\]
A direct computation then shows that 
\[
\alpha=\sum_{k\ge0}k!u^k\alpha_{k+1}
\]
satisfies $(b-uB)(\alpha)=0$. 

\subsubsection{Proof of non-degeneracy}\label{absnondeg}

We want to prove that $\alpha_1=\frac12(x^{-1}\otimes x-x\otimes x^{-1})$ is non-degenerate. 
First observe that the class of $x^{-1}\otimes x$ equals the one of $-x\otimes x^{-1}$ (and thus, the one of $\alpha_1$) in $\cat{HH}_1(\mathcal A)\simeq \Omega^1_{\mathcal A}$. 
Indeed, in (cohomological) degree $-1$ the Hochschild homology of ${\mathcal A}$ is $\Omega^1_{\mathcal A}=k[x^{\pm1}]d_{dR}x$. 
The class of a $x^{-1}\otimes x$, resp.~$-x\otimes x^{-1}$, is computed \textit{via} the Hochschild--Kostant--Rosenberg 
(HKR) map $a\otimes b\mapsto ad_{dR}b$, and we find
\[
x^{-1}d_{dR}x\,,\qquad\textrm{resp.}\quad -xd_{dR}(x^{-1})=xx^{-2}d_{dR}x=x^{-1}d_{dR}x\,.
\]
Hence it is sufficient to prove that the Hochschild $1$-cycle $x^{-1}\otimes x$ is non-degenerate. 

\medskip

The reduced Bar resolution of ${\mathcal A}$ is given by  
\[
\bar{\mathrm{B}}({\mathcal A})=\bigoplus_{n\geq0}\big({\mathcal A}\otimes\bar {\mathcal A}^{\otimes n}\otimes {\mathcal A}\big)[n]
\]
with the usual differential being given by an alternating sum of products of successive elements. 
We also have a smaller resolution 
\[
\mathrm{R}({\mathcal A})={\mathcal A}^e[1]\oplus {\mathcal A}^e
\]
with differential sending $1\otimes 1$ to $x\otimes 1-1\otimes x$. 
Its dual is 
\[
\mathrm{R}({\mathcal A})^\vee={\mathcal A}^e\oplus {\mathcal A}^e[-1]
\]
with the same formula for the differential. 

\medskip

There's a map $\mathrm{R}({\mathcal A})\to\bar{\mathrm{B}}({\mathcal A})$. In degree $0$ it is the identity, and in degree $-1$ it is given by 
$f\otimes g\mapsto f\otimes x\otimes g$. Using this smaller resolution we obtain the ``small Hochschild complex'': 
\[
{\mathcal A}[1]\oplus {\mathcal A}
\]
with zero differential. It maps inside the standard Hochschild complex as follows: in degree $0$ it is the identity, and in degree $-1$ it sends $f$ to $f\otimes x$. 
(this map is in fact a quasi-inverse to the HKR quasi-isomorphism). 
In the small Hochschild complex, the class of interest reads as $x^{-1}$, and one can show that as a map 
\begin{equation}
\label{eqn:iso}
\mathrm{R}({\mathcal A})^\vee[1]\longrightarrow \mathrm{R}({\mathcal A})
\end{equation}
it is nothing but the product with $x^{-1}\otimes 1$ (in both degrees), which is an isomorphism of complexes. 

\begin{remark}
We could have proven non-degeneracy first, and then use \cite[Proposition 5.7]{TdV-VdB} (see also Proposition~\ref{prop: lift} above) in order to obtain the existence (and uniqueness) of a cyclic lift. 
\end{remark}

\subsubsection{Yet another description of the Calabi--Yau structure on $k[x^{\pm1}]$}

For every $n$-Calabi--Yau category $\mathcal A$, with Calabi--Yau structure $c$, one can consider the same category with opposite Calabi--Yau structure $-c$, and denote it $\bar{\mathcal A}$. 
Then the functor $\mathcal{A}\coprod \bar{\mathcal A}\to\mathcal A$ is relative Calabi--Yau. 

\medskip

Let $k=ke$ be the terminal dg-category ($e$ denotes the identity of the single object); it is obviously $0$-Calabi--Yau, with Calabi--Yau structure being $e$. 
\begin{proposition}\label{prop: CYpushout}
There is an equivalence 
\[
k[x^{\pm1}]\simeq k\underset{k\coprod\bar{k}}\coprod k
\]
of $1$-Calabi--Yau dg-categories, where the Calabi--Yau structure on the left-hand-side is $\alpha$, and the one on the right-hand-side is obtained as a Calabi--Yau push-out. 
\end{proposition}
\begin{proof}
First of all we introduce the interval dg-category $kI$: it is the $k$-linearization of the category $I=1\tilde\longrightarrow 2$ with two objects and an isomorphism $x$ between them. 
Observe that we have a factorization $k\coprod k\to kI\to k$, where the first functor is a cofibration (the inclusion into $kI$ of its subcategory of objects), and the second functor is a trivial fibration. 
Hence our homotopy push-out can be computed as the strict push-out $k\underset{k\coprod k}\coprod kI\simeq k[x^{\pm1}]$. We thus get the requested equivalence of dg-categories. 
It remains to prove that the $1$-Calabi--Yau structures coincide. Thanks to Proposition \ref{prop: lift}, it is sufficient to prove that the underlying Hochschild classes coincide. 

Finally, the $0$-cycle $e_1-e_2$ is homotopic to zero in the cofibrant replacement $kI$ of $k$: $e_1- e_2=b(\alpha_1)$, where $\alpha_1$ still makes sense for $kI$. 
\end{proof}

\subsection{Relative Calabi--Yau structures on evaluations $k[x^{\pm1}]\to k$}\label{subsec-eval}

The pull-back of the closed $1$-form $\alpha=d_{dR}\log(x)$ along any $k$-point 
\[
q:{\mathcal A}=k[x^{\pm1}]\longrightarrow k
\]
of $\mathbb{G}_m$ (i.e.~$q\in k^\times$) obviously vanishes. This tells us that the morphism $q$ is relative pre-Calabi--Yau in the 
sense of \cite{BCS}\footnote{We warn again the reader that pre-Calabi--Yau in the sense of~\cite{BCS} (see also~\cite{ToCY}) is the non-commutative analog of pre-symplectic, and differs from the pre-Calabi--Yau notion from~\cite{IKV} 
that is the non-commutative analog of a Poisson structure. }. 
\begin{lemma}\label{eval}
The above relative pre-Calabi--Yau structure is non-degenerate. 
\end{lemma}
\begin{proof}
One first observes that both $(\mathrm{R}({\mathcal A})^\vee[1])\underset{{\mathcal A}^e}{\otimes}k^e$ and $\mathrm{R}({\mathcal A})\underset{{\mathcal A}^e}{\otimes}k^e$ are 
isomorphic to $k[1] \oplus k$ with zero differential. Moreover, after applying $\underset{{\mathcal A}^e}{\otimes}k^e$, the 
isomorphism \eqref{eqn:iso} becomes the multiplication by $q^{-1}$ on each component. 
Then recall that $k^\vee=k$, so that the morphism 
\[
k^\vee[1]=k[1]\to k[1]\oplus k =(\mathrm{R}({\mathcal A})^\vee[1])\underset{{\mathcal A}^e}{\otimes}k^e\,,
\quad\textrm{resp.}\quad
\mathrm{R}({\mathcal A})\underset{{\mathcal A}^e}{\otimes}k^e=k[1]\oplus k \to k
\]
is the obvious inclusion, resp.~projection. 
Hence the map 
\[
k^\vee[1]\longrightarrow \mathrm{fib}\left(\mathrm{R}({\mathcal A})\underset{{\mathcal A}^e}{\otimes}k^e\to k\right)
\]
identifies with the map 
\[
k[1]\overset{q^{-1}}{\longrightarrow} k[1]\simeq \mathrm{fib}\big(k[1]\oplus k\to k\big)\,.
\]
This proves the non-degeneracy. 
\end{proof}

\subsection{A Calabi--Yau cospan}

Our aim is to prove that the cospan 
\begin{equation}\label{equation-cospan}
k[x^{\pm1}]\coprod k[y^{\pm1}] \longrightarrow k\langle x^{\pm1},y^{\pm1}\rangle  \longleftarrow k[z^{\pm1}]\,,
\end{equation}
where the rightmost map is $z\mapsto xy$, is relative Calabi--Yau in the sense of \cite{BCS}.

Set $\beta_1=\frac{1}{2}(y^{-1}\otimes x^{-1}\otimes xy-y\otimes y^{-1}x^{-1}\otimes x)$, which satisfies\[
\alpha_1(xy)-(\alpha_1(x)+\alpha_1(y))=b(\beta_1).
\]

\begin{lemma}\label{lemma: CY}
The above homotopy $\beta_1$ is non-degenerate, and thus defines an almost $1$-Calabi--Yau structure on the cospan \eqref{equation-cospan}. This almost $1$-Calabi--Yau structure lifts uniquely to a $1$-Calabi--Yau structure. 
\end{lemma}
\begin{proof}
As a preliminary observation, let us recall on the one hand that ${\mathcal B}:=k\langle x^{\pm1},y^{\pm1}\rangle$ also has a 
small resolution as a ${\mathcal B}$-bimodule: 
\[
\mathrm{R}({\mathcal B})=({\mathcal B}^e)^{\oplus2}[1] \oplus {\mathcal B}^e
\]
with differential sending $(1\otimes1,0)$ to $x\otimes 1-1\otimes x$, and $(0,1\otimes1)$ to $y\otimes 1-1\otimes y$. 
Therefore 
\[
\mathrm{R}({\mathcal B})^\vee={\mathcal B}^e\oplus ({\mathcal B}^e)^{\oplus2}[-1]
\]
with differential sending $1\otimes 1$ to $(x\otimes 1-1\otimes x,y\otimes 1-1\otimes y)$. 

As the maps $\alpha_1(xy)$ and $\alpha_1(x)+\alpha_1(y)$ are homotopic via $\beta_1$, the following diagram is homotopy commutative
\[
\xymatrix{ {\mathcal B}^\vee[1] \ar[d]\ar[r] &  {\mathcal A}^\vee \underset{{\mathcal A}^e}{\otimes}{\mathcal B}^e[1] \overset{\alpha_1(xy)}{\simeq}  {\mathcal A}\underset{{\mathcal A}^e}{\otimes}{\mathcal B}^e \ar[d]  \\
( {\mathcal A}^{\oplus2})^\vee \underset{{\mathcal A}^e}{\otimes} {\mathcal B}^e [1] \overset{\alpha_1(x) +\alpha_1(y)}{\simeq} {\mathcal A}^{\oplus 2} \underset{{\mathcal A}^e}{\otimes}{\mathcal B}^e\ar[r] & {\mathcal B}
}
\]
where ${\mathcal A}= k[x^{\pm1}]$.
Following~\S\ref{absnondeg},  ${\mathcal A}\underset{{\mathcal A}^e}{\otimes}{\mathcal B}^e \simeq {\mathcal B}^e[1] \oplus {\mathcal B}^e$, with differential sending $1\otimes1$ to $x \otimes 1 -1\otimes x$. 
Hence, we get that the fibre of the map 
\[ ({\mathcal B}^e[1])^{\oplus 3}  \oplus ({\mathcal B}^e)^{\oplus 3} \to ({\mathcal B}^e)^{\oplus2}[1] \oplus {\mathcal B}^e\]
induced by $\alpha_1(xy)- \alpha_1(x) + \alpha_1(y)$  is isomorphic to $\mathrm{R}({\mathcal B})^\vee[1]={\mathcal B}^e[1]\oplus ({\mathcal B}^e)^{\oplus2}$. 
Then, using Theorem \ref{theorem-astuce}, we get that $\beta_1$ lifts to a unique homotopy $\beta$ between $\alpha(xy)$ and $\alpha(x)+\alpha(y)$. 
Therefore the cospan \eqref{equation-cospan} carries a  $1$-Calabi--Yau structure. Below we give an alternative presentation of this cospan. 
\end{proof}

\subsubsection{Another description of the Calabi--Yau cospan}\label{subsection: pants}

Observe that we have the following (strict) commuting diagram in the category $(\cat{Cat}_k^{sm})_{k/\cat{HC}^-}$ of smooth dg-categories equipped with a negative cyclic $0$-cycle (in order to lighten the notation, we omit coproducts): 
\[
\xymatrix{
&{\begin{matrix}k & k \end{matrix}}\ar[d]&&{\begin{matrix}k \\ k\end{matrix}}\ar[d] & \\
\varnothing \ar[ru]\ar[rd]& {\begin{matrix}k & k \end{matrix}} & {\begin{matrix}k & \bar{k} \\ \bar{k} & k\end{matrix}} \ar[ur]\ar[ul]\ar[dr]\ar[dl]& k &\varnothing \ar[lu]\ar[ld]\\
&{\begin{matrix}k & k \end{matrix}}\ar[u]&&{\begin{matrix}k & k \end{matrix}}\ar[u]&
}
\]
It admits a replacement by a (homotopy coherent) commuting diagram in the $\infty$-category $(\scat{Cat}_k^{sm})_{k/\cat{HC}^-}$: 
\[
\xymatrix{
&{\begin{matrix}kI & k\bar{I} \end{matrix}}\ar[d]&&{\begin{matrix}k \\ k\end{matrix}}\ar[d] & \\
\varnothing \ar[ru]\ar[rd]& {\begin{matrix}kI & k\bar{I} \end{matrix}} & {\begin{matrix}k & \bar{k} \\ \bar{k} & k\end{matrix}} \ar[ur]\ar[ul]\ar[dr]\ar[dl]& k & \varnothing \ar[lu]\ar[ld]\\
&{\begin{matrix}kI & k\bar{I} \end{matrix}}\ar[u]&&{\begin{matrix}k & k \end{matrix}}\ar[u]&
}
\]
Observe that the above diagram strictly commutes in $\cat{Cat}_k$, but that the negative cyclic $0$-cycles only match up to homotopy. 

By composing horizontal cospans we obtain a new (homotopy) commuting diagram in the $\infty$-category $(\scat{Cat}_k^{sm})_{k/\cat{HC}^-}$: 
\[\xymatrix{
& k[z^{\pm1}] \ar[d] & \\
\varnothing \ar[ur]\ar[dr] & k\langle x^{\pm1},y^{\pm1}\rangle & \varnothing \ar[ul]\ar[dl]\\
& k[x^{\pm1}] \coprod k[y^{\pm1}]  \ar[u]&
}
\]
One can finally see that the vertical cospan of dg-categories coincides with \eqref{equation-cospan}. Additionnaly, we have $1$-cycles $c_z$ and $c_{x,y}$ in $\cat{HC}^-(k[z^{\pm1}])$ and $\cat{HC}^-(k[x^{\pm1}]\coprod k[y^{\pm1}])$, respectively, 
together with a homotopy $c_{x,y,z}$ between their images in $ \cat{HC}^-(k\langle x^{\pm1},y^{\pm1}\rangle)$. 
\begin{proposition}
The triple $(c_{x,y},c_z,c_{x,y,z})$ defines a $1$-Calabi--Yau structure on \eqref{equation-cospan}, that coincides with the one from Lemma~\ref{lemma: CY}. 
\end{proposition}
\begin{proof}
As we have already seen in Proposition~\ref{prop: CYpushout}, the $1$-Calabi--Yau structures on $k[x^{\pm1}]\coprod k[y^{\pm1}]$ match up: $c_{x,y}\sim\alpha(x)+\alpha(y)$. 

They also match on $k[z^{\pm1}]$, but there is a subtlety that is worth noticing. As usual, according to the uniqueness of cyclic lifts from Proposition~\ref{prop: lift}, in order to prove that $c_z\sim\alpha(z)$ 
it is sufficient to prove that $c_z^\natural\sim\alpha(z)^\natural=\alpha_1(z)$. Now, computed strictly, the top horizontal push-out gives the $k$-linearization ${\mathcal C}$ of a category with 
two objects $1,2$ and two isomorphisms $x:1\tilde\to 2$ and $y:2\tilde\to 1$. Of course, we have an equivalence $k[z^{\pm1}]\tilde\to {\mathcal C}$, sending $z$ to $xy$. Following a similar calculation as in the proof of 
Proposition~\ref{prop: CYpushout}, we get on ${\mathcal C}$ the Hochschild $1$-cycle $\alpha_1(x)+\alpha_1(y)$. Up to a Hochschild boundary, this matches up with the image of $\alpha_1(z)$ through the equivalence given by $z\mapsto xy$. 
Indeed, the formula for the homotopy $\beta_1$ still makes sense in ${\mathcal C}$. 

It remains to prove that the homotopy $c_{x,y,z}$ matches with $\beta$. As the underlying Hochschild homotopy $\beta_1=\beta^{\natural}$ is non-degenerate (thanks to Lemma~\ref{lemma: CY}), 
according to the uniqueness of cyclic lifts from Theorem~\ref{theorem-astuce}, it suffices to prove that the underlying Hochschild homotopies $c_{x,y,z}^\natural$ and $\beta^\natural=\beta_1$ coincide. 
We already proved it, as $\beta_1$ is the homotopy that identifies $\alpha_1(x)+\alpha_1(y)$ with $\alpha_1(z)$ in ${\mathcal C}$. 
\end{proof}


\begin{remark}\label{remark-conjecturale}
Let us put what we have done so far in a more general perspective, by first recalling from~\cite[\S5.1]{BD1} that for a closed oriented $d$-manifold $M$, the $k$-linearization $\mathfrak{L}(M):=\cat{dg}(\cat{Sing}(M))$ of the fundamental $\infty$-groupoid of $M$ 
carries a $d$-Calabi--Yau structure. Moreover, in \textit{loc.~cit.} the authors also prove that if $N$ is a compact oriented $(d+1)$-manifold with boundary $\partial N=M$, then one gets a $d$-Calabi--Yau structure on the natural functor 
$\mathfrak{L}(M)\to \mathfrak{L}(N)$. We conjecture the existence of a symmetric monoidal $(\infty,n)$-category $\scat{CY}_n^s$ of $n$-iterated $s$-Calabi--Yau cospans, similar to the iterated category of lagrangian correspondences sketched 
in~\cite{CalTFT}, and rigoroulsy constructed in~\cite{CHS}. We also conjecture that the functor $\cat{dg}(\cat{Sing}(-))$ leads to a fully extended oriented TFT in every dimension: i.e.~it should admit an upgrade to a symmetric monoidal $(\infty,n)$-functor 
\[
\scat{Bord}_n^{or}\longrightarrow \scat{CY}_n^0
\]
for every $n$ (in particular, $k$ is $n$-dualizable in $\scat{CY}_n^0$). 
For the above presentation of the Calabi--Yau cospan structure on $k\langle x^{\pm1},y^{\pm1}\rangle$, we took inspiration from a construction of the pair-of-pants as a suitable composition of $2$-iterated oriented bordisms (see 
subsection~\ref{subsec-fusionpants} below, where this decomposition of the pair-of-pants is made explicit), and guessed the diagram one shall write by pretending that the conjecture was known. 
\end{remark}

\begin{remark}
The Calabi--Yau push-out of this cospan with the evaluation Calabi--Yau morphism $q:k[z^{\pm1}]\to k$ from~\S\ref{subsec-eval} gives the Calabi--Yau cospan associated with the Calabi--Yau isomorphism $k[x^{\pm1}]\to k[y^{\pm1}]$ given by $x\mapsto q^{-1}y$ (see Remark~\ref{rem:inv}). 
More precisely, the push-out gives a morphism 
\[
k[x^{\pm1}]\coprod k[y^{\pm1}] \longrightarrow k\langle x^{\pm1},y^{\pm1}\rangle/(xy=q)\,,
\]
under which the image of $\alpha_1(x)+\alpha_1(y)$ is identically $0$. Then using that $inv$ gives an isomorphism between the Calabi--Yau structure and its inverse on $k[x^{\pm1}]$ (see Remark~\ref{rem:inv}, again), we obtain the desired Calabi--Yau cospan from $k[x^{\pm1}]$ to $k[y^{\pm1}]$. 
\end{remark}


\section{Multiplicative preprojective algebras}\label{section-preproj}

Consider a quiver $Q$, which consists in a vertex set $V$, and an oriented edge set $E$: to each edge $e$ we associate a source $s(e)$ and a target $t(e)$ in $V$. We consider its double version $\overline Q=(V,\overline E=E\sqcup E^*)$, where $E^*$ consists in reverse arrows $e^*:t(e)\to s(e)$, and extend ${}^*$ in an involution of $\overline E$ by setting $e^{**}=e$ for every $e\in E$. We also set $\epsilon(e)=1$ and $\epsilon(e^*)=-1$ for all $e\in E$.
As mentioned in the introduction, 
Crawley-Boevey and Shaw introduced in~\cite{CBS}
the \emph{multiplicative preprojective algebra} $\Lambda^q(Q)$, where $q\in (k^*)^V$. It is given as the quotient of a localization of $k\overline Q$ by the relation\[
\prod_{e\in\overline E}(1+ee^*)^{\epsilon(e)}-\sum_{v\in V}q_ve_v\]
where $e_v$ denotes the length $0$ idempotent path at $v$.
It is thus required to invert all $1+ee^*$ for $e\in E^*$, which actually amounts to inverting $1+ee^*$ for all $e\in\overline E$. We denote $k\overline{Q}_{loc}$ the localization of $k\overline{Q}$ with respect to these elements. 

The definition of $\Lambda^q(Q)$ a priori requires an ordering on $\overline E$, but the resulting quotient actually doesn't depend on it (up to isomorphism~\cite[Theorem 1.4]{CBS}).

\begin{remark}\label{setquivfree}
We can either view $\Lambda^q(Q)$ as an algebra, or as a category (with objects the vertices of $Q$, that correspond to the idempotents of $\Lambda^q(Q)$). 
There is a Morita morphism from one to another, so that it doesn't matter for what we do (see~\cite[Remark 5.4]{BCS}).  
\end{remark}

\subsection{Relative Calabi--Yau structure for the $A_2$ quiver}

Consider the quiver $A_2=(V=\{1,2\},E=\{e:1\to2\})$, with orthogonal idempotents $e_1$ and $e_2$ satisfying $1=e_1+e_2$, and write \[a_1=e_1+e^*e\text{ and }a_2=e_2+ee^*.\]
Note that\[
1+e^*e\text{ invertible}\Leftrightarrow a_1 \text{ invertible}\Leftrightarrow a_2 \text{ invertible}\Leftrightarrow 1+ee^*\text{ invertible},\]
in which case \begin{align}\label{invrels}\begin{split}
(1+e^*e)^{-1}&=e_2+a_1^{-1}\\
a_2^{-1}&=e_2-ea_1^{-1}e^*\\
a_1^{-1}&=e_1-e^*a_2^{-1}e\\
(1+ee^*)^{-1}&=e_1+a_2^{-1}.\end{split}\end{align}
Thus in the $A_2$ case, the product in the multiplicative preprojective relation reads
\[
(1+ee^*)(1+e^*e)^{-1}=a_2+a_1^{-1}=(1+e^*e)^{-1}(1+ee^*).\]
Denote by ${\mathcal B}$ the localization $k\overline{A_2}[a_{1}^{-1},a_2^{-1}]$, and define morphisms $\mu_i:k[x_i^{\pm1}]\rightarrow {\mathcal B}$, $i\in\{1,2\}$, by setting\[
\mu_1(x_1)=a_1^{-1}\quad \text{and}\quad \mu_2(x_2)=a_2.
\]
Equalities~\eqref{invrels} further imply \begin{align}\label{invrels2}\begin{split}
a_2^{-1}e&=ea_1^{-1}\\
e^*a_2^{-1}&=a_1^{-1}e^*\\
e^*ea_1^{-1}&=e_1-a_1^{-1}\\
ee^*a_2^{-1}&=e_2-a_2^{-1}.\end{split}\end{align}

\subsubsection{The homotopy}

Note that $\mu_1$ maps $\alpha_1$ to $\frac{1}{2}(a_1\otimes a_1^{-1}-a_1^{-1}\otimes a_1)$, and $\mu_2$ maps $\frac{1}{2}(\alpha_1$ to $a_2^{-1}\otimes a_2-a_2\otimes a_2^{-1})$ where tensor products are performed over the algebra ${\mathcal R}=\oplus_{v\in V}ke_v$.
Thus\begin{align*}
2(\mu_1(\alpha_1)+\mu_2(\alpha_1))&=e^*e\otimes a_1^{-1}-a_1^{-1}\otimes e^*e+a_2^{-1}\otimes ee^*-ee^*\otimes a_2^{-1}\\
&\qquad\qquad+e_1\otimes a_1^{-1}-a_1^{-1}\otimes e_1+a_2^{-1}\otimes e_2-e_2\otimes a_2^{-1}\\
&=e^*e\otimes a_1^{-1}-a_1^{-1}\otimes e^*e+a_2^{-1}\otimes ee^*-ee^*\otimes a_2^{-1}\\
&\qquad\qquad+1\otimes (a_1^{-1}- a_2^{-1})
\end{align*}
as $a_i^{-1}\otimes e_i=a_i^{-1}\otimes1=0$ in the normalized Hochschild complex.
Direct computations, helped by~\eqref{invrels2}, show that \[
e^*e\otimes a_1^{-1}-a_1^{-1}\otimes e^*e+a_2^{-1}\otimes ee^*-ee^*\otimes a_2^{-1}\]
 is the image under $b$ of\[	
e^*\otimes e\otimes a_1^{-1}+a_1^{-1}\otimes e^*\otimes e- e^*\otimes a_2^{-1}\otimes e-a_2^{-1}\otimes e\otimes e^*.\]
Also,\begin{align*}
1\otimes (a_1^{-1}- a_2^{-1})&=1 \otimes (a_1^{-1}-e_1-a_2^{-1}+e_2)&&[\text{normalization}]\\
&=1\otimes (ee^*a_2^{-1}-e^*ea_1^{-1})&&[\eqref{invrels2}]\\
&=-Bb(e^*\otimes ea_1^{-1})&&[\eqref{invrels2}]\\
&=bB(e^*\otimes e a_1^{-1}).
\end{align*}
Hence $\mu_1(\alpha_1)+\mu_2(\alpha_1)$ is the image under $b$ of \begin{align}\label{homotopyA2}\begin{split}
\beta_1&=\frac{1}{2}\left(e^*\otimes e\otimes a_1^{-1}+a_1^{-1}\otimes e^*\otimes e- e^*\otimes a_2^{-1}\otimes e-a_2^{-1}\otimes e\otimes e^*+B(e^*\otimes e a_1^{-1})\right)\\
&=\frac{1}{2}\left(e^*\otimes e\otimes \mu+\mu\otimes e^*\otimes e- e^*\otimes \mu^{-1}\otimes e-\mu^{-1}\otimes e\otimes e^*\right.\\
&\qquad\qquad\left.+1\otimes e^*\otimes e \mu-1\otimes \mu^{-1}e\otimes e^*\right)\end{split}\end{align}
if $\mu=\mu_1(x_1)+\mu_2(x_2)$.

\subsubsection{Non-degeneracy}

  \begin{proposition}\label{a2nondeg}
  The cospan $\mu_1\amalg-\mu_2$ carries an almost $1$-Calabi--Yau structure, that lifts uniquely to a $1$-Calabi--Yau structure thanks to Theorem~\ref{theorem-astuce}.
   \end{proposition}
  \begin{proof} 
 Set ${\mathcal A}= k[x_1^{\pm1}] \amalg k[x_2^{\pm1}]$ and $u=\mu_1\amalg-\mu_2$. Thanks to the existence of the homotopy $\beta_1$ given by~\eqref{homotopyA2}, the following diagram homotopy commutes: 
 \begin{equation}\label{hocoA2}
\xymatrix{ {\mathcal B}^\vee [1] \ar[r]^-{u^\vee[1]} \ar[d]& {\mathcal A}^\vee [1]   \underset{{\mathcal A}^e}{\otimes} {\mathcal B}^e  \ar[r] \ar[d]^-{\alpha_1 \underset{{\mathcal A}^e}{\otimes} {\mathcal B}^e} &\mathrm{cofib}(u^\vee[1])\ar[d] \\
\mathrm{fib}(u) \ar[r]& {\mathcal A} \underset{{\mathcal A}^e}{\otimes} {\mathcal B}^e  \ar[r]^-{u}& {\mathcal B} } 
\end{equation}
To show the non-degeneracy, we need to prove that the vertical maps in \eqref{hocoA2} are isomorphisms. Since ${\mathcal A}$ is $1$-Calabi--Yau, it is sufficient to prove that the leftmost vertical map is an isomorphism. 
Set ${\mathcal A}_i=k[x_i^{\pm1}]$, and ${\mathcal B}_i^e={\mathcal A}_i \underset{{\mathcal A_i}^e}{\otimes} {\mathcal B}^e$ induced by $\mu_i$.
Using the resolutions from~\S\ref{absnondeg}, we can replace $ {\mathcal A}  \underset{{\mathcal A}^e}{\otimes} {\mathcal B}^e$ with the complex 
\[({\mathcal B}^e_1\oplus {\mathcal B}_2^e)[1] \oplus  ({\mathcal B}_1^e\oplus {\mathcal B}_2^e)\]
 with differential \[
 d:(p_1\otimes q_1,p_2\otimes q_2) \mapsto (p_1a_1^{-1}\otimes q_1-p_1\otimes a_1^{-1}q_1,p_2a_2\otimes q_2-p_2\otimes a_2q_2)\] 
 where $p_i,q_i\in {\mathcal B}$. 
 A ${\mathcal B}$-bimodule resolution of ${\mathcal B}$ is given by \[
 \Omega^1({\mathcal B}) \stackrel{d'} \longrightarrow {\mathcal B} ^e.\]
  By \cite[Therorem 10.6]{Schofield} (see also~\cite[Remark 5.4]{BCS}), we can identify $\Omega^1({\mathcal B})$ with 
${\mathcal B} \underset{\mathcal R}{\otimes} k\overline{E} \underset{\mathcal R}{\otimes} {\mathcal B}$  and $d'(1\otimes v\otimes 1)=v\otimes 1-1\otimes v$, where $\mathcal R$ still denotes $\oplus_{v\in V}ke_v$. 
Here for $A_2$ the edge set $E$ is simply $\{e\}$. Hence, $u$ is given by the following commutative diagram
\[ \xymatrix{
{\mathcal B}^e_1\oplus {\mathcal B}_2^e\ar[r]^-{f} \ar[d]_-d & {\mathcal B} \underset{\mathcal R}{\otimes} k\overline{E} \underset{\mathcal R}{\otimes} {\mathcal B} \ar[d]^-{d'} \\
{\mathcal B}^e_1\oplus {\mathcal B}_2^e  \ar[r]_-{\tau} & {\mathcal B} \underset{\mathcal R}{\otimes} {\mathcal B}. } 
\]
where \begin{align*}
f(p_1\otimes q_1,p_2\otimes q_2) &=f_1(p_1\otimes q_1)-f_2(p_2\otimes q_2)\\
\tau(p_1\otimes q_1,p_2\otimes q_2) &=p_1\otimes q_1-p_2\otimes q_2.
\end{align*}
 Let us give a concrete description of $f$. We have a $k$-linear map $\iota: k\overline E \to  {\mathcal B} \underset{\mathcal R}{\otimes} k\overline{E}\underset{\mathcal R}{\otimes} {\mathcal B}$ which 
sends a path $p=\alpha_1 \cdots \alpha_n$, $\alpha_i \in \overline{E}$, to \[
\sum_{i=1}^n \alpha_1 \cdots \alpha_{i-1} \otimes \alpha_i  \otimes \alpha_{i+1} \cdots \alpha_n .\]
 This map has a natural ${\mathcal B}^e$-linear extension  ${\mathcal B} \to {\mathcal B} \underset{\mathcal R}{\otimes} k\overline{E}\underset{\mathcal R}{\otimes} {\mathcal B}$, still denoted by $\iota$, satisfying\begin{equation}\label{iotadiff}
\iota(bb')=b\iota(b')+\iota(b)b'.
\end{equation}
Then it can be checked that the maps $f_i:{\mathcal B}_i^e \to  {\mathcal B} \underset{\mathcal R}{\otimes} k\overline{E}\underset{\mathcal R}{\otimes} {\mathcal B}$ are given as ${\mathcal B}^e$-linear maps by \begin{align*}
f_1(1\otimes1) &= \iota(a^{-1})\\
f_2(1\otimes 1)&=-\iota(a_2).\end{align*}
We then identify $\mathrm{fib}(u)$ with $(f,\tau)$.

The resolution of ${\mathcal B}^\vee $ as a ${\mathcal B}^e$-module is given by
\[d'^\vee: {\mathcal B}\underset{\mathcal R}{\otimes}   {\mathcal B} \to {\mathcal B} \underset{\mathcal R}{\otimes} k \overline{E} \underset{\mathcal R}{\otimes} {\mathcal B} , \ 1 \otimes e_i \otimes 1 \mapsto \sum_{\alpha \in \overline{E}}( \alpha \otimes \alpha^* \otimes 1- 1 \otimes \alpha^* \otimes \alpha).\]
In the $A_2$ case, this just reads\[
d'^\vee(1\otimes 1)=e\otimes e^*\otimes 1-1\otimes e^*\otimes e+e^*\otimes e\otimes 1-1\otimes e\otimes e^*.\]

 The equivalence \[
 g=\alpha_1 \underset{{\mathcal A}^e}{\otimes} {\mathcal B}^e: ({\mathcal B}^e_1\oplus {\mathcal B}_2^e)[1] \oplus  ({\mathcal B}_1^e\oplus {\mathcal B}_2^e)  \rightarrow ({\mathcal B}^e_1\oplus {\mathcal B}_2^e)[1] \oplus  ({\mathcal B}_1^e\oplus {\mathcal B}_2^e)\]
 is induced by the image of $\alpha_1$ under $\mu_1\amalg-\mu_2$, hence by the internal product $m$ with \[
(a_1\otimes1-1\otimes a_1,a_2^{-1}\otimes 1-1\otimes a_2^{-1})\]
  on both terms, thanks to~\S\ref{absnondeg}. 
The homotopy $\beta_1$ defined by~\eqref{homotopyA2} induces a zero homotopy $h$ of the map \[\xymatrix{
{\mathcal B}^\vee[1] \ar[rr]^-{ug u^\vee[1]} && {\mathcal B}}.\] 
With the chosen resolutions, this yields a map 
$h: {\mathcal B}\underset{\mathcal R}{\otimes} \overline{E} \underset{\mathcal R}{\otimes} {\mathcal B} \to {\mathcal B}\underset{\mathcal R}{\otimes} \overline{E} \underset{\mathcal R}{\otimes} {\mathcal B}$ such that the triangles in the following diagram commute 
\[ \xymatrix{
{\mathcal B} \underset{\mathcal R}{\otimes} {\mathcal B} \ar[r]^-{d'^\vee} \ar[d]_-{fm\tau^\vee} & {\mathcal B} \underset{\mathcal R}{\otimes} k\overline{E} \underset{\mathcal R}{\otimes} {\mathcal B} \ar[d]^-{ \tau mf^\vee} \ar[ld]^-h\\
{\mathcal B} \underset{\mathcal R}{\otimes} k\overline{E} \underset{\mathcal R}{\otimes} {\mathcal B} \ar[r]_-{d'} & {\mathcal B} \underset{\mathcal R}{\otimes} {\mathcal B}. } 
\]
where $\tau^\vee(1\otimes1)=(1\otimes1,-1\otimes1)$.
Now, 
\begin{align*}
fm\tau^\vee(1\otimes1)&=fm(1\otimes1,-1\otimes1)\\
&=f(a_1\otimes1-1\otimes a_1,1\otimes a_2^{-1}-a_2^{-1}\otimes 1)\\
&=a_1\iota(a_1^{-1})-\iota(a_1^{-1})a_1+\iota(a_2)a_2^{-1}-a_2^{-1}\iota(a_2)\\
&=-\iota(a_1)a_1^{-1}+a_1^{-1}\iota(a_1)+\iota(a_2)a_2^{-1}-a_2^{-1}\iota(a_2)\\
&=-(e^*\otimes e\otimes 1+1\otimes e^*\otimes e)a_1^{-1}+a_1^{-1}(e^*\otimes e\otimes 1+1\otimes e^*\otimes e)\\
&\qquad\qquad +(e\otimes e^*\otimes 1+1\otimes e\otimes e^*)a_2^{-1}-a_2^{-1}(e\otimes e^*\otimes 1+1\otimes e\otimes e^*)\\
&=hd'^\vee(1\otimes1)
\end{align*}
if (using~\eqref{invrels2})\begin{align*}
h(1\otimes e\otimes 1)&=a_2^{-1}\otimes e\otimes 1-1\otimes e\otimes a_1^{-1}\\
h(1\otimes e^*\otimes 1)&=1\otimes e^*\otimes a_2^{-1}-a_1^{-1}\otimes e^*\otimes 1.
\end{align*}
The homotopy $h$ therefore induces an isomorphism ${\mathcal B}^\vee[1] \overset{\sim}{\longrightarrow} \mathrm{fib}(u)$ as wished (it is the leftmost vertical map in~\eqref{hocoA2}).
\end{proof}

\subsection{Fusion}\label{subsection-fusion}

Following~\cite{VdB}, we use a fusion procedure to go from the $A_2$ case to the case of an arbitrary quiver $Q=(V,E)$. 
The following endows the ``noncommutative group-valued'' moment map for $k\overline{Q}_{loc}=k\overline {Q}[(1+ee^*)^{-1}]_{e\in\overline E}$, that defines the multiplicative preprojective algebra, with a Calabi--Yau structure. 

\begin{theorem}\label{fuz}
There is a 1-Calabi--Yau structure on the morphism
\begin{eqnarray*}
\mu:\coprod_{v\in V}k[z_v^{\pm1}] & \longrightarrow & k\overline {Q}_{loc} \\
z_v & \longmapsto & \prod_{e\in E\cap t^{-1}(v)}(1+ee^*)\times \prod_{e\in E\cap s^{-1}(v)}(1+e^*e)^{-1}\,.
\end{eqnarray*}
 \end{theorem}

\begin{proof}
Denote by $Q^\mathrm{sep}$ the quiver with same edge set $E$ but vertex set $\overline E=\{v_e=s(e),v_{e^*}=t(e)\}$. It is the disjoint union of $|E|$ copies of $A_2$ that we aim to glue by ``fusing'' vertices. This will be done using composition of Calabi--Yau structures by means of push-outs.
Thanks to~\S\ref{a2nondeg}, we have a 1-Calabi--Yau morphism\begin{equation}\label{CYsep}
\coprod_{e\in E}(k[x_e^{\pm1}]\amalg k[y_e^{\pm1}])\rightarrow k\overline {Q^\mathrm{sep}}[(1+ee^*)^{-1}]_{e\in\overline E}
\end{equation}
given by $x_e\mapsto (e_{s(e)}+e^*e)^{-1}$ and $y_e\mapsto e_{t(e)}+ee^*$.

For each vertex $v\in V$, fix a total ordering of all edges of $E$ with target $v$ and the same with $E^*$.
Consider $e,f=e+1\in E$, both with target $v$. We have a 1-Calabi--Yau cospan~\ref{equation-cospan} given by\[
k[y_e^{\pm1}]\coprod k[y_f^{\pm1}] \rightarrow k\langle y_e^{\pm1} ,y_f^{\pm1}\rangle\leftarrow k[z_{e,f}^{\pm1}]\]
with $z_{e,f}\mapsto y_ey_f$. 
Similarly, if $e^*,f^*=e^*+1\in E^*$, both with target $v$, we have a 1-Calabi--Yau cospan given by\[
k[x_e^{\pm1}]\coprod k[x_f^{\pm1}] \rightarrow k\langle x_e^{\pm1} ,x_f^{\pm1}\rangle\leftarrow k[z_{e,f}^{\pm1}]\]
with $z_{e,f}\mapsto x_ex_f$.
Finally, if $e=\max_E t^{-1}(v)$ and $f^*=\min_{E^*}t^{-1}(v)$, we have a 1-Calabi--Yau cospan given by\[
k[y_e^{\pm1}]\coprod k[x_f^{\pm1}] \rightarrow k\langle y_e^{\pm1} ,x_f^{\pm1}\rangle\leftarrow k[z_{e,f}^{\pm1}]\]
with $z_{e,f}\mapsto y_ex_f$. Proceeding to ordered compositions of cospans, we get a 1-Calabi--Yau cospan given by\[
\xymatrix{
&C_v:=k\Big\langle(y_e^{\pm1})_{e\in E\cap t^{-1}(v)},(x_e^{\pm1})_{e\in E\cap s^{-1}(v)}\Big\rangle&\\
\left(\coprod_{e\in E\cap t^{-1}(v)}k[y_e^{\pm1}]\right)\coprod\left(\coprod_{e\in E\cap s^{-1}(v)}  k[x_e^{\pm1}]\right)\ar[ur]\!\!\!\!\!\!\!\!\!\!\!\!\!\!\!\!\!\!\!\!\!\!\!\!\!\!\!\!\!\!\!\!\!\!\!\!\!\!\!\!\!\!\!\!\!\!\!\!\!\!\!\!\!&&
k[z_v^{\pm1}]\ar[ul]
}\]
where coproducts and variables are ordered.

Now fix an ordering on $V$, composing the above yields a cospan\[
\coprod_{e\in E}(k[x_e^{\pm1}]\amalg k[y_e^{\pm1}])\rightarrow 
\coprod_{v\in V}C_v\leftarrow\coprod_{v\in V}k[z_v^{\pm1}]
\]
that can be composed with~\eqref{CYsep} in order to get a 1-Calabi--Yau structure on $\mu$ as expected.
\end{proof}

\begin{remark} Note that this proof is independent of the choice of the function $\epsilon:\overline Q\to\{\pm1\}$ defining the preprojective multiplicative algebra.\end{remark}

\subsection{Reduction}

Consider a family of 1-Calabi--Yau morphisms $q_v:k[z_v^{\pm1}]\to k$, $v\in V$; that is a collection $q=(q_v)_{v\in V}\in(k^\times)^V$. 
Thanks to Lemma~\ref{eval} and Theorem~\ref{fuz}, we have a $2$-Calabi--Yau structure on the push-out of $\mu$ with $\coprod_{v\in V}q_v$. 
To compute this push-out, let us use for each $v$ the $k[z_v^{\pm1}]$-cofibrant replacement of $k$ given by $k\langle {z'_v},z_v^{\pm1}\rangle$ where $z'_v$ lies in degree $-1$, $z_v$ in degree $0$ 
and the differential is given by $z'_v\mapsto z_v-q_v$. We thus get the following. 
\begin{theorem}\label{dgmult}
For every $q\in (k^\times)^{V(Q)}$, there is a $2$-Calabi--Yau structure on the dg-algebra $\Upsilon^q(Q)$ defined as follows: 
\begin{itemize}
\item As a graded algebra, $\Upsilon^q(Q)$ is freely generated over $k\overline {Q}_{loc}$ by the bimodule 
\[
\left(k\overline {Q}_{loc}\right)^e\underset{\mathcal{R}^e}{\otimes}\left(\underset{v\in V}{\oplus}kz'_v\right)\,;
\] 
\item The differential sends $z'_v$ to 
\[
\left(\prod_{e\in E\cap t^{-1}(v)}(1+ee^*)\times \prod_{e\in E\cap s^{-1}(v)}(1+e^*e)^{-1}\right)-q_v\,.
\]
\end{itemize}
\end{theorem}

\begin{remark}
\begin{itemize}
\item The zeroth cohomology of $\Upsilon^q(Q)$ is the deformed preprojective algebra $\Lambda^q(Q)$.
\item The dg-algebra $\Upsilon^q(Q)$ coincides with the one of~\cite[\S5.C]{BK} (in the case of a nodal curve with rationnal components), as well as the one of~\cite[Definition 4.3]{KaSc}
\item Theorem~\ref{dgmult} generalizes~\cite[Theorem 5.52]{Yeung} from star-shaped quivers to arbitrary ones.
\end{itemize}
\end{remark}


\section{Comparison: moduli of objects} \label{section-comparison}

The moduli of objects $\mathbf{Perf}$ was introduced by To\"en-Vaqui\'e in \cite{ToVa} as a functor \[
\mathbf{Perf}: \scat{Cat}_k^{f.t.} \to \scat{dSt}_k^{Art}\]
 from the $\infty$-category of finite type dg-categories to the $\infty$-category of derived Artin $k$-stacks. For a finite type dg-category $\mathcal{A}$ and a commutative differential graded $k$-algebra $B$, 
$\mathbf{Perf}_{\mathcal{A}}(B):=\cat{Map}_{\scat{Cat}_k}(\cat{A},\cat{Mod}_B^{perf})$ consists in perfect $B$-module valued $\mathcal{A}$-modules. 
In~\cite{PTVV}, $n$-shifted symplectic structures for Artin stack, as well as $n$-shifted lagrangian morphisms and correspondences (see also~\cite{CalTFT}) have been introduced.
Calabi--Yau structures on dg-categories and functors can be considered as non-commutative analogs of shifted symplectic and lagrangian structures in the following sense: 
by \cite[Theorem 5.5]{BD2} (see also \cite{ToCY}), the moduli stack of objects $\mathbf{Perf}$ sends $n$-Calabi--Yau structures to $(2-n)$-shifted symplectic structures, and can be extended to a functor from
$n$-Calabi--Yau cospans to $(2-n)$-shifted lagrangian correspondences. 

Another way of producing new shifted symplectic and lagrangian structures from old ones was discovered in~\cite[Theorem 2.5]{PTVV}: 
it is shown that for an $n$-shifted symplectic Artin stack $X$, the mapping stack $\textsc{Map}\big(-,X\big)$ in $\scat{dSt}^{Art}$ sends (nice enough) $d$-oriented Artin stacks to $(n-d)$-shifted symplectic stacks. 
By \cite[Theorem 4.8]{CalTFT} the functor $\textsc{Map}\big(-,X\big)$ sends (nice enough) $d$-oriented cospans to $(n-d)$-shifted lagrangian correspondences. 
Note that the Betti-stack functor, denoted by $(-)_B$, maps $d$-oriented manifolds to (sufficiently nice) $d$-oriented derived stacks. 

\subsection{Moduli of objects of $k[x^{\pm1}]$, derived loop stacks, and the adjoint quotient}\label{subsection-comparison}

On the one hand, the $1$-Calabi--Yau structure on $k[x^{\pm1}]$ as constructed in Section \ref{section: CY} induces a $1$-shifted symplectic structure on the derived stack $\mathbf{Perf}_{ k[x^{\pm1}]}$. 
On the other hand, $\mathbf{Perf}_{ k[x^{\pm1}]}$ is equivalent to the derived loop stack $\mathcal L\mathbf{Perf}_{k}:=\textsc{Map}(B\mathbb{Z},\mathbf{Perf}_{k})$. 
Knowing that $B\mathbb{Z}\simeq S^1_B$ is $1$-oriented, and that $\mathbf{Perf}_{k}$ is $2$-shifted symplectic (because $k$ is $0$-Calabi--Yau), we obtain, thanks to \cite[Theorem 2.5]{PTVV},  
a transgressed $1$-shifted symplectic structure on $\mathcal L\mathbf{Perf}_{k}$. 
\begin{proposition}
There is an equivalence 
\[
\mathbf{Perf}_{k[x^{\pm1}]}\simeq\mathcal L\mathbf{Perf}_{k}
\]
as $1$-shifted symplectic derived stacks. 
\end{proposition}
\begin{proof}
On the one hand, recall that for every $n$-shifted symplectic derived stack $X$, the derived loop stack $\mathcal L X$ is equivalent, as an $(n-1)$-shifted symplectic derived stack, to the derived lagrangian intersection 
\[
X\underset{X\times\overline{X}}{\times}X\,,
\]
where $\overline{X}$ denotes the same derived stack equipped with the opposite $n$-shifted symplectic structure. Indeed, the functor $\textsc{Map}\big((-)_B,X\big)$ is an oriented topological field theory 
(see \cite[Theorem 4.8]{CalTFT}), and as such it sends the gluing of two oriented manifolds along a common boundary to the corresponding derived lagrangian intersection. The case of interest for us is the one of $S^1$, 
that is obtained by gluing two closed intervals along two points: 
\[
S^1\simeq \mathrm{pt}\underset{\mathrm{pt}\coprod \overline{\mathrm{pt}}}{\coprod}\mathrm{pt}\,,
\]
where $\overline{\mathrm{pt}}$ denotes the point with its opposite orientation. 

On the other hand, using Proposition \ref{prop: CYpushout} and the fact that $\mathbf{Perf}$ sends compositions of Calabi--Yau cospans to compositions of lagrangian correspondences (and, in particular, 
Calabi--Yau pushouts to lagrangian intersections), see \cite{BD2} and \cite[\S 6.1.2]{BCS}, we obtain that 
\[
\mathbf{Perf}_{k[x^{\pm1}]}\simeq \mathbf{Perf}_{k\underset{k\coprod \bar k}{\coprod}k}\simeq \mathbf{Perf}_k\underset{\mathbf{Perf}_k\times\overline{\mathbf{Perf}}_k}{\times}\mathbf{Perf}_k\simeq \mathcal{L}\mathbf{Perf}_k
\]
as $1$-shifted symplectic derived stacks. 
\end{proof}

Finally, by restricting ourselves to the open substack consisting of perfect modules of amplitude $0$ and fixed dimension $n$, we get back the transgressed $1$-shifted symplectic structure on $\mathcal L(BGL_n)$
 (recall that the open embedding $BGL_n\hookrightarrow \mathbf{Perf}_k$ is a $2$-shifted symplectomorphism). 
According to \cite{Safronov}, this $1$-shifted symplectic structure coincides with the explicit one given on the adjoint quotient $[GL_n/GL_n]\simeq \mathcal L(BGL_n)$ by the quasi-hamiltonian formalism (see \cite{CalTFT,Safronov}). 

\begin{remark}
The $1$-Calabi--Yau evaluation morphism $q:k[x^{\pm1}]\to k$, $q\in k^\times$, induces a $1$-shifted lagrangian morphism $\mathbf{Perf}_k\to \mathbf{Perf}_{k[x^{\pm1}]}$. We let the reader check that, when restricted on the open substacks of amplitude $0$ 
modules of dimension $n$, it gives back the lagrangian morphism $BGL_n\to [GL_n/GL_n]$ corresponding to the group-valued moment map $\mathrm{pt}\to GL_n$ given by $q\mathrm{Id}_n$. 
\end{remark}

\subsection{Moduli of objects of $k\langle x^{\pm1},y^{\pm1}\rangle$, pair of pants, and fusion}\label{subsec-fusionpants}

Recall the lagrangian structure on the correspondence
\begin{equation}\label{equation-lagrangianpant}
    \mathbf{Perf}_{k[x^{\pm1}]\coprod k[y^{\pm1}]} \longleftarrow \mathbf{Perf}_{k\langle x^{\pm1},y^{\pm1}\rangle}  \longrightarrow \mathbf{Perf}_{k[z^{\pm1}]}\,,
\end{equation}
given by applying the moduli of objects $\mathbf{Perf}$ to the Calabi--Yau cospan~\eqref{equation-cospan} (see~\cite[Theorem 5.5]{BD2}). 
Using the other description from~\S\ref{subsection: pants} of the Calabi--Yau cospan~\eqref{equation-cospan}, and the fact that the functor $\textsc{Map}\big((-)_B,\mathbf{Perf}\big)$ sends $\mathrm{pt}$ to $\mathbf{Perf}_k$, 
we obtain an alternative construction of the lagrangian correspondence~\eqref{equation-lagrangianpant}. This is achieved by applying $\textsc{Map}\big((-)_B,\mathbf{Perf}\big)$ to the diagram

\[
\xymatrix{
&{\begin{matrix}\mathrm{pt} & \mathrm{pt} \end{matrix}}\ar[d]&&{\begin{matrix}\mathrm{pt} \\ \mathrm{pt}\end{matrix}}\ar[d] & \\
\emptyset \ar[ru]\ar[rd]& {\begin{matrix}\mathrm{pt} & \mathrm{pt} \end{matrix}} & {\begin{matrix}\mathrm{pt} & \bar{\mathrm{pt}} \\ \bar{\mathrm{pt}} & \mathrm{pt}\end{matrix}} \ar[ur]\ar[ul]\ar[dr]\ar[dl]& \mathrm{pt} & \emptyset \ar[lu]\ar[ld]\\
&{\begin{matrix}\mathrm{pt} & \mathrm{pt} \end{matrix}}\ar[u]&&{\begin{matrix}\mathrm{pt} & \mathrm{pt} \end{matrix}}\ar[u]&
}
\]
and then horizontally compose correspondences, as $\mathbf{Perf}$ sends push-outs to pull-backs. Here we recall that $\overline{(-)}$ denotes the orientation, respectively the symplectic structure, with inverted sign. 
A convenient replacement of the above diagram looks as follows:

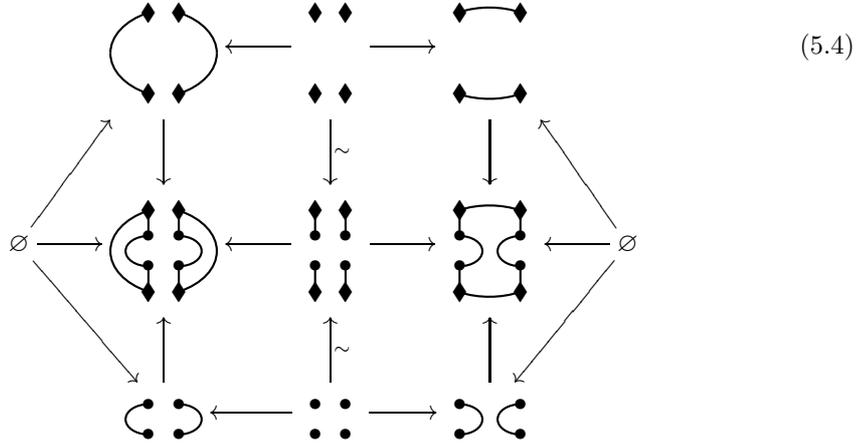
\begin{equation}\label{pushpants}
\xymatrix{
&
\begin{tikzpicture}[scale=0.2, thick, baseline=(current bounding box.center)]
\begin{scope}[xshift=-1cm]
\draw(0,2.7)node{\small$\blacklozenge$};
\draw(0,-2.7)node{\small$\blacklozenge$};
\draw (-2.5,0) arc (180:117:4.5 and 3);
\draw (-2.5,0) arc (180:243:4.5 and 3);
\end{scope}
\begin{scope}[xshift=1cm]
\draw(0,2.7)node{\small$\blacklozenge$};
\draw(0,-2.7)node{\small$\blacklozenge$};
\draw (2.5,0) arc (0:63:4.5 and 3);
\draw (2.5,0) arc (0:-63:4.5 and 3);
\end{scope}
\end{tikzpicture}\ar[d]
&
\begin{tikzpicture}[scale=0.2, thick, baseline=(current bounding box.center)]
\begin{scope}[xshift=-1cm]
\draw(0,2.7)node{\small$\blacklozenge$};
\draw(0,-2.7)node{\small$\blacklozenge$};
\end{scope}
\begin{scope}[xshift=1cm]
\draw(0,2.7)node{\small$\blacklozenge$};
\draw(0,-2.7)node{\small$\blacklozenge$};
\end{scope}
\end{tikzpicture}
\ar[l]\ar[d]^-[left]{\sim}\ar[r]&
\begin{tikzpicture}[scale=0.2, thick, baseline=(current bounding box.center)]
\draw(-2,-2.7)node{\small$\blacklozenge$};
\draw(2,-2.7)node{\small$\blacklozenge$};
\draw(-2,2.7)node{\small$\blacklozenge$};
\draw(2,2.7)node{\small$\blacklozenge$};
\draw (2,2.7) arc (63:116:4.5 and 3);
\draw (2,-2.7) arc (-63:-116:4.5 and 3);
\end{tikzpicture}
\ar[d]
\\
\varnothing\ar[ur]\ar[dr]\ar[r]&
\begin{tikzpicture}[scale=0.2, thick, baseline=(current bounding box.center)]
\begin{scope}[xshift=-1cm]
\draw(0,2.7)node{\small$\blacklozenge$};
\draw(0,-2.7)node{\small$\blacklozenge$};
\draw(0,-1)node{$\bullet$};
\draw(0,1)node{$\bullet$};
\draw (0,-2.7)--(0 ,-1)(0,1)--(0,2.7);
\draw (0, 1) arc (90:270:1.5 and 1);
\draw (-2.5,0) arc (180:117:4.5 and 3);
\draw (-2.5,0) arc (180:243:4.5 and 3);
\end{scope}
\begin{scope}[xshift=1cm]
\draw(0,2.7)node{\small$\blacklozenge$};
\draw(0,-2.7)node{\small$\blacklozenge$};
\draw(0,-1)node{$\bullet$};
\draw(0,1)node{$\bullet$};
\draw (0,-2.7)--(0 ,-1)(0,1)--(0,2.7);
\draw (0,-1) arc (-90:90:1.5 and 1);
\draw (2.5,0) arc (0:63:4.5 and 3);
\draw (2.5,0) arc (0:-63:4.5 and 3);
\end{scope}
\end{tikzpicture}&
\begin{tikzpicture}[scale=0.2, thick, baseline=(current bounding box.center)]
\begin{scope}[xshift=-1cm]
\draw(0,2.7)node{\small$\blacklozenge$};
\draw(0,-2.7)node{\small$\blacklozenge$};
\draw(0,-1)node{$\bullet$};
\draw(0,1)node{$\bullet$};
\draw (0,-2.7)--(0 ,-1)(0,1)--(0,2.7);
\end{scope}
\begin{scope}[xshift=1cm]
\draw(0,2.7)node{\small$\blacklozenge$};
\draw(0,-2.7)node{\small$\blacklozenge$};
\draw(0,-1)node{$\bullet$};
\draw(0,1)node{$\bullet$};
\draw (0,-2.7)--(0 ,-1)(0,1)--(0,2.7);
\end{scope}
\end{tikzpicture}
\ar[l]\ar[r]&
\begin{tikzpicture}[scale=0.2, thick, baseline=(current bounding box.center)]
\draw(-2,-2.7)node{\small$\blacklozenge$};
\draw(2,-2.7)node{\small$\blacklozenge$};
\draw(-2,2.7)node{\small$\blacklozenge$};
\draw(2,2.7)node{\small$\blacklozenge$};
\draw(-2,-1)node{$\bullet$};
\draw(2,-1)node{$\bullet$};
\draw(-2,1)node{$\bullet$};
\draw(2,1)node{$\bullet$};
\draw (-2, -1) arc (-90:90:1.5 and 1);
\draw (2,1) arc (90:270:1.5 and 1);
\draw (2,2.7) arc (63:116:4.5 and 3);
\draw (2,-2.7) arc (-63:-116:4.5 and 3);
\draw (-2,-2.7)--(-2 ,-1)(-2,1)--(-2,2.7)(2,-2.7)--(2 ,-1)(2,1)--(2,2.7);
\end{tikzpicture}
&\varnothing\ar[ul]\ar[l]\ar[dl]
\\
&
\begin{tikzpicture}[scale=0.2, thick, baseline=(current bounding box.center)]
\begin{scope}[xshift=-1cm]
\draw(0,-1)node{$\bullet$};
\draw(0,1)node{$\bullet$};
\draw (0, 1) arc (90:270:1.5 and 1);
\end{scope}
\begin{scope}[xshift=1cm]
\draw(0,-1)node{$\bullet$};
\draw(0,1)node{$\bullet$};
\draw (0,-1) arc (-90:90:1.5 and 1);
\end{scope}
\end{tikzpicture}
\ar[u]&
\begin{tikzpicture}[scale=0.2, thick, baseline=(current bounding box.center)]
\begin{scope}[xshift=-1cm]
\draw(0,-1)node{$\bullet$};
\draw(0,1)node{$\bullet$};
\end{scope}
\begin{scope}[xshift=1cm]
\draw(0,-1)node{$\bullet$};
\draw(0,1)node{$\bullet$};
\end{scope}
\end{tikzpicture}
\ar[l]\ar[u]_-[right]{\sim}\ar[r]&
\begin{tikzpicture}[scale=0.2, thick, baseline=(current bounding box.center)]
\draw(-2,-1)node{$\bullet$};
\draw(2,-1)node{$\bullet$};
\draw(-2,1)node{$\bullet$};
\draw(2,1)node{$\bullet$};
\draw (-2, -1) arc (-90:90:1.5 and 1);
\draw (2,1) arc (90:270:1.5 and 1);
\end{tikzpicture}
\ar[u]
}
\end{equation}

Taking pushouts along the three horizontal correspondences above yields the following $1$-oriented cospan/cobordism:

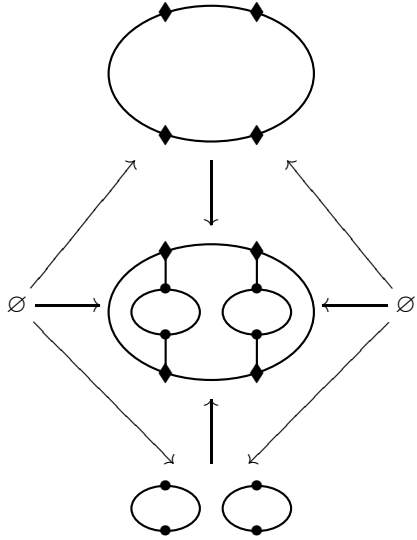
\[
\xymatrix{
&
\begin{tikzpicture}[scale=0.3, thick, baseline=(current bounding box.center)]
\draw (4.5,0) arc (0:360:4.5 and 3);
\draw(-2,-2.7)node{\small$\blacklozenge$};
\draw(2,-2.7)node{\small$\blacklozenge$};
\draw(-2,2.7)node{\small$\blacklozenge$};
\draw(2,2.7)node{\small$\blacklozenge$};
\end{tikzpicture}
\ar[d]
\\
\varnothing\ar[ur]\ar[dr]\ar[r]&
\begin{tikzpicture}[scale=0.3, thick, baseline=(current bounding box.center)]
\draw (-2, 1) arc (90:270:1.5 and 1);
\draw (-2, -1) arc (-90:90:1.5 and 1);
\draw (2,1) arc (90:270:1.5 and 1);
\draw (2,-1) arc (-90:90:1.5 and 1);
\draw (4.5,0) arc (0:360:4.5 and 3);
\draw (-2,-2.7)--(-2 ,-1)(-2,1)--(-2,2.7)(2,-2.7)--(2 ,-1)(2,1)--(2,2.7);
\draw(-2,-2.7)node{\small$\blacklozenge$};
\draw(2,-2.7)node{\small$\blacklozenge$};
\draw(-2,2.7)node{\small$\blacklozenge$};
\draw(2,2.7)node{\small$\blacklozenge$};
\draw(-2,-1)node{$\bullet$};
\draw(2,-1)node{$\bullet$};
\draw(-2,1)node{$\bullet$};
\draw(2,1)node{$\bullet$};
\end{tikzpicture}
&\varnothing\ar[ul]\ar[l]\ar[dl]
\\
&
\begin{tikzpicture}[scale=0.3, thick, baseline=(current bounding box.center)]
\draw (-2, 1) arc (90:270:1.5 and 1);
\draw (-2, -1) arc (-90:90:1.5 and 1);
\draw (2,1) arc (90:270:1.5 and 1);
\draw (2,-1) arc (-90:90:1.5 and 1);
\draw(-2,-1)node{$\bullet$};
\draw(2,-1)node{$\bullet$};
\draw(-2,1)node{$\bullet$};
\draw(2,1)node{$\bullet$};
\end{tikzpicture}\ar[u]&
}
\]

Note that the manifold at the center of the diagram is the pair of pants (see Figure~\ref{figure}).
\begin{figure}[h!]
\[
\begin{tikzpicture}[scale=0.4, thick, baseline=(current bounding box.center)]

\draw (-2, 1) arc (90:270:1.5 and 1);
\draw (-2, -1) arc (-90:90:1.5 and 1);
\draw (2,1) arc (90:270:1.5 and 1);
\draw (2,-1) arc (-90:90:1.5 and 1);

\draw(-0.52,-0.2) .. controls (0.5, 2.8) and (1.5, 2.8) .. (0.52,0.2);

\draw(-3.53,0.15)-- (-1.51,5.2);

\draw(0,6) arc(90:170: 1.5 and 1);

\draw (3.5, 0) -- (5.5, 5);
\draw  (0, 6) -- (4,6)  ;

\draw (4,6) arc (90:0: 1.5 and 1);

\draw[dashed] (-1.5,5) arc (-180:-90: 1.5 and 1)--(4,4) arc (-90:-30: 1.5 and 1);

\end{tikzpicture}
\hspace{0.2cm}\leftrightsquigarrow\hspace{0.2cm}
\begin{tikzpicture}[scale=0.4, thick, baseline=(current bounding box.center)]

\draw[color=red]  (-2, -1) --(-1.25, 0.875);
\draw[color=red] (2.75, 0.875) -- (2,-1);

\draw[color=red] (2,1) -- (4,6) ;

\draw (-2, 1) arc (90:270:1.5 and 1);
\draw (-2, -1) arc (-90:90:1.5 and 1);
\draw (2,1) arc (90:270:1.5 and 1);
\draw (2,-1) arc (-90:90:1.5 and 1);

\draw(-0.52,-0.2) .. controls (0.5, 2.8) and (1.5, 2.8) .. (0.52,0.2);

\draw(-3.53,0.15)-- (-1.51,5.2);

\draw(0,6) arc(90:170: 1.5 and 1);

\draw (3.5, 0) -- (5.5, 5);
\draw[color=red] (-2, 1) -- (0, 6);
\draw  (0, 6) -- (4,6)  ;

\draw (4,6) arc (90:0: 1.5 and 1);

\draw[dashed,color=red] (-1.25, 0.875) --(0, 4)  (4,4) -- (2.75, 0.875);
\draw[dashed] (-1.5,5) arc (-180:-90: 1.5 and 1)--(4,4) arc (-90:-30: 1.5 and 1);

\end{tikzpicture}
\hspace{0.2cm}\leftrightsquigarrow\hspace{0.2cm}
\begin{tikzpicture}[scale=0.4, thick, baseline=(current bounding box.center)]

\draw[dashed] (-1.25, 0.875) -- (0, 4) -- (3.2,4);

\draw  (3.2, 4) -- (4,4) -- (2,-1);

\draw(-1.25, 0.875) -- (-2, -1) arc (-90:90:1.5 and 1) -- (0, 6) -- (4,6)-- (2,1) arc (90:270:1.5 and 1);

\draw (-0.52,-0.2) .. controls (0.5, 2.8) and (1.5, 2.8) .. (0.52,0.2);

\begin{scope}[xshift=1.5cm]
\draw (2,-1) -- (2.75,0.875);
\draw[dashed] (2,-1) -- (4,4);

\draw (4,6) -- (2,1);

\draw (2,-1) arc (-90:90:1.5 and 1);

\draw (3.48, -0.2) -- (5.5, 5);

\draw (4,6) arc (90:0: 1.5 and 1);
\draw[dashed] (4,4) arc (-90:-30: 1.5 and 1);

\end{scope}

\begin{scope}[xshift=-1.5cm]

\draw (-2, 1) arc (90:270:1.5 and 1);
\draw (0,6) arc (90:170: 1.5 and 1);
\draw[dashed] (0,4) arc (-90:-180: 1.5 and 1);
\draw (0,4) arc (-90:-120: 1.5 and 1);

\draw(-3.53,0.15)-- (-1.51,5.2);

\draw (-2,1)--(and 0,6);

\draw (-2,1)--(and 0,6);
\draw (-2,-1)--(and 0,4);

\end{scope}

\end{tikzpicture}
\]

\[
\begin{tikzpicture}[scale=0.4, thick, baseline=(current bounding box.center)]

\draw (-2, 1) arc (90:270:1.5 and 1);
\draw (-2, -1) arc (-90:90:1.5 and 1);
\draw (2,1) arc (90:270:1.5 and 1);
\draw (2,-1) arc (-90:90:1.5 and 1);

\draw (4.5,0) arc (0:360:4.5 and 3);
\end{tikzpicture}
\hspace{0.2cm}\leftrightsquigarrow\hspace{0.2cm}
\begin{tikzpicture}[scale=0.4, thick, baseline=(current bounding box.center)]

\draw (-2, 1) arc (90:270:1.5 and 1);
\draw (-2, -1) arc (-90:90:1.5 and 1);
\draw (2,1) arc (90:270:1.5 and 1);
\draw (2,-1) arc (-90:90:1.5 and 1);

\draw (4.5,0) arc (0:360:4.5 and 3);

\draw[color=red] (-2,-2.7)--(-2 ,-1)(-2,1)--(-2,2.7)(2,-2.7)--(2 ,-1)(2,1)--(2,2.7);

\draw(-2,-2.7)node{\small$\blacklozenge$};
\draw(2,-2.7)node{\small$\blacklozenge$};
\draw(-2,2.7)node{\small$\blacklozenge$};
\draw(2,2.7)node{\small$\blacklozenge$};

\draw(-2,-1)node{$\bullet$};
\draw(2,-1)node{$\bullet$};
\draw(-2,1)node{$\bullet$};
\draw(2,1)node{$\bullet$};

\end{tikzpicture}
\hspace{0.2cm}\leftrightsquigarrow\hspace{0.2cm}
\begin{tikzpicture}[scale=0.4, thick, baseline=(current bounding box.center)]

\draw(-2,-2.7)node{\small$\blacklozenge$};
\draw(2,-2.7)node{\small$\blacklozenge$};
\draw(-2,2.7)node{\small$\blacklozenge$};
\draw(2,2.7)node{\small$\blacklozenge$};

\draw(-2,-1)node{$\bullet$};
\draw(2,-1)node{$\bullet$};
\draw(-2,1)node{$\bullet$};
\draw(2,1)node{$\bullet$};

\draw (-2, -1) arc (-90:90:1.5 and 1);
\draw (2,1) arc (90:270:1.5 and 1);

\draw (2,2.7) arc (63:116:4.5 and 3);
\draw (2,-2.7) arc (-63:-116:4.5 and 3);

\draw (-2,-2.7)--(-2 ,-1)(-2,1)--(-2,2.7)(2,-2.7)--(2 ,-1)(2,1)--(2,2.7);

\begin{scope}[xshift=-1.5cm]

\draw(-2,2.7)node{\small$\blacklozenge$};
\draw(-2,-2.7)node{\small$\blacklozenge$};

\draw(-2,-1)node{$\bullet$};
\draw(-2,1)node{$\bullet$};

\draw (-2,-2.7)--(-2 ,-1)(-2,1)--(-2,2.7);
\draw (-2, 1) arc (90:270:1.5 and 1);
\draw (-4.5,0) arc (180:117:4.5 and 3);
\draw (-4.5,0) arc (180:243:4.5 and 3);

\end{scope}

\begin{scope}[xshift=1.5cm]
\draw(2,2.7)node{\small$\blacklozenge$};
\draw(2,-2.7)node{\small$\blacklozenge$};

\draw(2,-1)node{$\bullet$};
\draw(2,1)node{$\bullet$};

\draw (2,-2.7)--(2 ,-1)(2,1)--(2,2.7);
\draw (2,-1) arc (-90:90:1.5 and 1);
\draw (4.5,0) arc (0:63:4.5 and 3);
\draw (4.5,0) arc (0:-63:4.5 and 3);

\end{scope}

\end{tikzpicture}
\]
\caption{The decomposition of the pair of pants corresponding to~\eqref{pushpants}, in 3d and 2d.}\label{figure}
\end{figure}
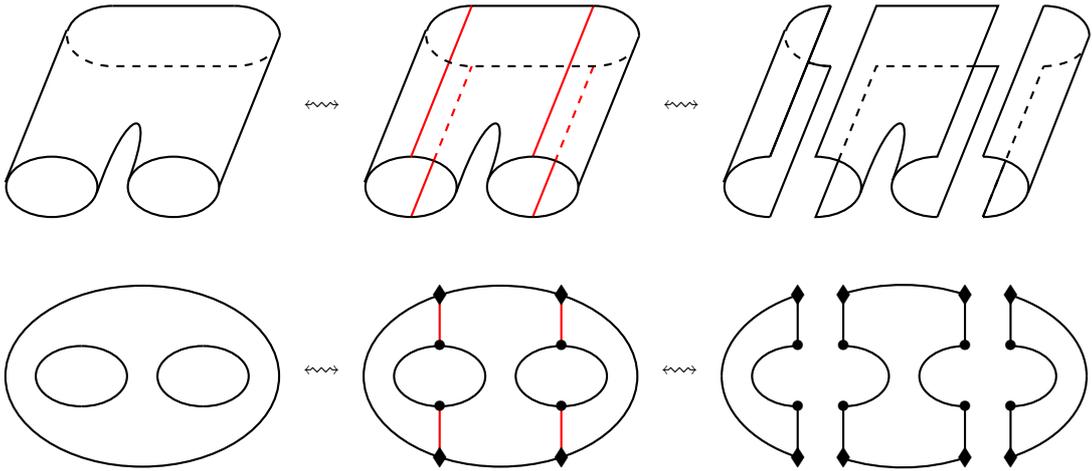

\medskip

Using that $\textsc{Map}\big((-)_B,\mathbf{Perf}\big)$ is a fully extended TFT~\cite{CHS}, we have that the lagrangian correspondence~\eqref{equation-lagrangianpant} is obtained by applying 
$\textsc{Map}\big((-)_B,\mathbf{Perf}\big)$ to the oriented cobordism given by the pair of pants, see~\cite[Theorem 4.8]{CalTFT}. 

Hence, when restricting ourselves to the substacks of amplitude zero modules of fixed dimension $n$, we get a lagrangian correspondence 
\[
[GL_n/GL_n]\times[GL_n/GL_n]\longleftarrow [(GL_n\times GL_n)/GL_n]\longrightarrow [GL_n/GL_n]
\]
that coincides with the one given by applying $\textsc{Map}\big((-)_B,BGL_n\big)$ to the pair of pants (and using~\cite[Theorem 4.8]{CalTFT}). It was shown by Safronov~\cite{Safronov} that composition with this lagrangian correspondence 
gives back the fusion procedure from~\cite{AMM}. 

\begin{remark}
Notice that $\mathbf{Perf}$ sends the conjectural fully dualizable object $k$ in $\scat{CY}_n^0$ from Remark~\ref{remark-conjecturale} to the fully dualizable object $\mathbf{Perf}_k$ in $\scat{Lag}_n^2$. 
As a consequence $\mathbf{Perf}$ shall intertwine the conjectural fully extended TFT from Remark~\ref{remark-conjecturale} with the fully extended TFT $\textsc{Map}\big((-)_B,\mathbf{Perf}_k\big)$ from~\cite{CHS} 
(see also~\cite{CalTFT} for a heuristic). What we have done above is following this guiding idea and applying it in an \textit{ad hoc} way to the case of the pair-of-pants.  
\end{remark}

\subsection{Open questions}

Before applying reduction, we have a $1$-shifted lagrangian structure on the morphism 
\[
\mathbf{Perf}_{k\overline{Q}_{loc}}\longrightarrow \mathbf{Perf}_{k[x^{\pm1}]}^{V(Q)}\,.
\]
Fixing a dimension vector $\vec{n}\in(\mathbb{Z}_{\geq0})^{V(Q)}$, one can consider the open substacks of dimension $\vec{n}$ amplitude $0$ modules. 
This leads to a $1$-shifted lagrangian structure on the morphism 
\[
\big[\mathbf{Rep}(k\overline{Q}_{loc},\vec{n})/GL_{\vec{n}}\big]\longrightarrow \big[GL_{\vec{n}}/GL_{\vec{n}}\big]
\]
(note that $\mathbf{Rep}(k\overline{Q}_{loc},\vec{n})\simeq \mathbf{DRep}(k\overline{Q}_{loc},\vec{n})$). 
Knowing from~\S\ref{subsection-comparison} that the $1$-shifted symplectic structure on the target is the standard one on the adjoint quotient, we obtain that $\mathbf{Rep}(k\overline{Q}_{loc},\vec{n})$ is a quasi-hamiltonian $GL_{\vec{n}}$-space. 
We conjecture that it coincides with the quasi-hamiltonian structure on the very same space from~\cite{VdB,VdB2,Yam,Boalch}. 

Observe that it suffices to prove the conjecture for the simplest case $Q=A_2$. Indeed, in the above references the general case is obtained from the $A_2$ one by the fusion process of~\cite{AMM}. 
We proceeded in the same way in~\S\ref{subsection-fusion}, and it follows from~\S\ref{subsec-fusionpants} that our fusion procedure coincides with the one of~\cite{AMM} for substacks of amplitude $0$ modules. 

\medskip

In order to prove the conjecture for $A_2$, one could try to prove a similar statement directly at the noncommutative level. To achieve this, one would first have to rigorously prove that Van den Bergh's noncommutative quasi-hamiltonian 
structures~\cite{VdB2} naturally lead to relative Calabi--Yau structures (as was argued in the introduction). 

\medskip

Finally, we believe that our $1$-Calabi--Yau structure on the dg-version of the multiplicative preprojective algebra shall give back the non-degenerate pairing appearing in~\cite{BK} (in the case of a nodal curve with rational components) on the one hand, 
and could probably be recovered from~\cite{ST} on the other hand. 
In both cases, it is very likely that the proof of the comparison will again go through a reduction to the $A_2$ quiver, using fusion.


\begin{thebibliography}{XXX}

\bibitem{AMM}A.~Alekseev, A.~Malkin \& E.~Meinrenken, 
Lie group valued moment maps, 
J. Differential Geom. \textbf{48} (1998), no. 3, 445--495.

\bibitem{Beil}A. A.~Beilinson,
How to glue perverse sheaves, 
K-theory, arithmetic and geometry, 
Lecture Notes in Math., \textbf{1289}, Springer, Berlin, 1987, 42--51.

\bibitem{BK}R.~Bezrukavnikov \& M.~Kapranov, 
Microlocal sheaves and quiver varieties, 
Annales de la Facult\'e des Sciences de Toulouse \textbf{25} (2016),  no. 2--3, 473--516. 

\bibitem{BoalchAnnals}P.~Boalch, 
Geometry and braiding of Stokes data; Fission and wild character varieties, 
Annals of Mathematics \textbf{179} (2014), no. 1, 301--365. 

\bibitem{Boalch}P.~Boalch, 
Global Weyl groups and a new theory of multiplicative quiver varieties, 
Geom. Topol. \textbf{19} (2015), no. 6, 3467--3536. 

\bibitem{BCS}T.~Bozec, D.~Calaque \& S.~Scherotzke, 
Relative critical loci and quiver moduli, 
preprint available at \url{https://arxiv.org/abs/2006.01069}, 
to appear in Annales Scientifiques de l'ENS. 

\bibitem{BD1}C.~Brav \& T.~Dyckerhoff, 
Relative Calabi--Yau structures, 
Compositio Math. \textbf{155} (2019), 372--412. 

\bibitem{BD2}C.~Brav \& T.~Dyckerhoff, 
Relative Calabi--Yau structures II: Shifted Lagrangians in the moduli of objects, 
Selecta Math. (N.S.) \textbf{27} (2021), no. 4, Paper No. 63. 

\bibitem{CalTFT}D.~Calaque, 
Lagrangian structures on mapping stacks and semi-classical TFTs, 
Contemporary Mathematics \textbf{643} (2015). 

\bibitem{CPTVV}D.~Calaque, T.~Pantev, B.~To\"en, M.~Vaqui\'e \& G.~Vezzosi, 
Shifted Poisson structures and deformation quantization, 
Journal of Topology \textbf{10} (2017), no. 2, 483--584. 

\bibitem{CHS}D.~Calaque, R.~Haugseng \& C.~Scheimbauer, 
On the AKSZ construction in derived algebraic geometry as an extended topological field theory, 
preprint available at \url{https://arxiv.org/abs/2108.02473}, 
to appear in Memoirs of the AMS. 

\bibitem{ChaFai}O.~Chalykh \& M.~Fairon,
Multiplicative quiver varieties and generalised Ruijsenaars-Schneider models,
J. Geom. Phys. \textbf{121}, (2017), 413--437.

%
\bibitem{CBS}W.~Crawley-Boevey \& P.~Shaw, 
Multiplicative preprojective algebras, middle convolution and the Deligne--Simpson problem, 
Adv. Math. \textbf{201} (2006), 180--208. 

\bibitem{FH}D.~Fern\'andez \& E.~Herscovich, 
Double quasi-Poisson algebras are pre-Calabi--Yau, 
Int. Math. Res. Not. IMRN \textbf{2022} (2022), no. 23, 18291--18345. 

\bibitem{GePo}I. M.~Gel'fand \& V. A.~Ponomarev, 
Model algebras and representations of graphs,
Funktsional. Anal. i Prilozhen. \textbf{13} no. 3, (1979), 1--12. 

\bibitem{IKV}N.~Iyudu, M.~Kontsevich \& Y.~Vlassopoulos, 
Pre-Calabi-Yau algebras as noncommutative Poisson structures, 
Journal of Algebra \textbf{567} (2021), no.~1, 63--90. 

\bibitem{KaSc}D.~Kaplan \& T.~Schedler, 
Multiplicative preprojective algebras are $2$-Calabi--Yau, 
Algebra \& Number Theory \textbf{17} (2023), no. 4, 831--883. 

\bibitem{KellerDG}B.~Keller, 
On differential graded categories,
Proc. ICM Vol. II, 151--190, Eur. Math. Soc., Z\"urich, 2006. 

\bibitem{Lusz}G.~Lusztig, G,
Quivers, perverse sheaves, and quantized enveloping algebras,
J. Amer. Math. Soc. \textbf{4} no. 2, (1991), 365--421. 

\bibitem{MS2}V.~Melani \& P.~Safronov, 
Derived coisotropic structures II: stacks and quantization, 
Selecta Math. \textbf{24} (2018), 3119--3173. 

\bibitem{Nak}H.~Nakajima, 
Instantons on ALE spaces, quiver varieties, and Kac-Moody algebras,
Duke Math. J. \textbf{76} no. 2, (1994), 365--416. 

\bibitem{PTVV}T.~Pantev, B.~To\"en, M.~Vaqui\'e \& G.~Vezzosi, 
Shifted symplectic structures, 
Publications math\'ematiques de l'IH\'ES \textbf{117} (2013), no. 1, 271--328. 

\bibitem{Prid}J. P.~Pridham, 
Shifted Poisson and symplectic structures on derived $N$-stacks, 
Joural of Topology \textbf{10} (2017), no. 1, 178--210. 

\bibitem{Ring}C. M.~Ringel, 
The preprojective algebra of a quiver, 
CMS Conf. Proc., \textbf{24}, Amer. Math. Soc., Providence, RI (1998), 467--480. 

\bibitem{Safronov}P.~Safronov, 
Quasi-Hamiltonian reduction via classical Chern-Simons theory, 
Advances in Mathematics \textbf{287} (2016), 733--773. 

\bibitem{ScTi}T.~Schedler \& A.~Tirelli, 
Symplectic resolutions for multiplicative quiver varieties and character varieties for punctured surfaces, 
in \textit{Representation theory and algebraic geometry -- a conference celebrating the birthdays of Sasha Beilinson and Victor Ginzburg}, pp. 393--459,
Trends Math., Birkhäuser/Springer, Cham, 2022. 

\bibitem{Schofield}A. H. Schofield, 
Representations of rings over skew fields, 
London Math. Soc. Lec. Note Ser. \textbf{92}, Cambridge Univ. Press, 1985.

\bibitem{ST}V.~Shende \& A.~Takeda, 
Calabi--Yau structures on topological Fukaya categories, 
preprint available at \url{https://arxiv.org/abs/1605.02721}. 

\bibitem{TdV-VdB}L.~de Thanhoffer de V\"olcsey \& M.~Van den Bergh, 
Calabi--Yau Deformations and Negative Cyclic Homology, 
Journal of noncommutative geometry \textbf{12} (2018), no. 4, 1255--1291. 

\bibitem{ToDG2}B.~To\"en, 
Lectures on dg-categories, 
in \textit{Topics in Algebraic and Topological $K$-Theory}, 
Lecture Notes in Mathematics \textbf{2008}, Springer, Berlin, Heidelberg. 

\bibitem{ToEMS}B.~To\"en, 
Derived algebraic geometry, 
EMS Surveys in Mathematical Sciences \textbf{1} (2014), no. 2, 153--240. 

\bibitem{ToCY}B.~To\"en, 
Structures symplectiques et de Poisson sur les champs en cat\'egories, 
preprint available at \url{https://arxiv.org/abs/1804.10444}. 

\bibitem{ToVa}B.~To\"en \& M.~Vaqui\'e, 
Moduli of objects in dg-categories, 
Annales de l'ENS \textbf{40} (2007), 387--444. 

\bibitem{VdB}M.~Van den Bergh,
Double Poisson algebras,
Trans. Amer. Math. Soc. \textbf{360} (2008), no. 11, 5711--5769.

\bibitem{VdB2}M.~Van den Bergh, 
Non-commutative quasi-Hamiltonian spaces, 
in \textit{Poisson geometry in mathematics and physics}, Contemp. Math. \textbf{450} (2008), 273--300.

\bibitem{Yam}D.~Yamakawa, 
Geometry of multiplicative preprojective algebra, 
Int. Math. Res. Pap. \textbf{2008} (2008). 

\bibitem{Yeung}W.-K.~Yeung, 
Relative Calabi--Yau completions, 
preprint available at \url{https://arxiv.org/abs/1612.06352v1}. 

\end{thebibliography}
\end{document}